\documentclass[11pt, leqno]{amsart}
\usepackage[svgnames, dvipsnames, table]{xcolor}
\usepackage[colorlinks = true, allcolors = Crimson, pagebackref=true, colorlinks]{hyperref}

\usepackage[utf8]{inputenc}
\usepackage{amscd}
\usepackage{soul}
\usepackage{multicol}
\usepackage{caption}
\usepackage{subcaption}
\usepackage{float}
\usepackage{longtable}
\usepackage{lipsum}
\usepackage{nomencl}
\usepackage{setspace}
\usepackage[ruled, linesnumbered, algosection]{algorithm2e}
\usepackage{pifont}
\usepackage{mathrsfs}
\usepackage{stmaryrd}
\usepackage{bm}
\usepackage{verbatim}
\usepackage{xcolor}
\usepackage{colortbl}
\usepackage{booktabs}
\usepackage{tabu}
\usepackage{color}
\usepackage{graphicx}
\usepackage{tikz,pgfplots,pgfplotstable}
\usepgfplotslibrary{fillbetween}
\usetikzlibrary{patterns}
\pgfplotsset{compat=1.16}
\usepackage{rotating}
\usepackage{diagbox}
\usepackage{amssymb}
\usepackage{epstopdf}
\usepackage{tikz}
\definecolor{UCRceleste}{RGB}{0,192,243}
\definecolor{Crimson}{RGB}{220, 20, 60}
\definecolor{mintgreen}{RGB}{152,255,152}
\definecolor{pinksalmon}{RGB}{255,102,102}
\definecolor{hueso}{RGB}{245,245,220}
\definecolor{marfil}{RGB}{255,253,208}
\definecolor{amarillo}{RGB}{255,255,0}
\usetikzlibrary{decorations.markings,arrows}
\usetikzlibrary{decorations.pathreplacing}
\usepackage[inner=1.0in,outer=1.0in,bottom=1.0in, top=1.0in]{geometry}

\numberwithin{equation}{section}
\newtheorem{theorem}{Theorem}[section]
\newtheorem{lemma}[theorem]{Lemma}
\newtheorem{proposition}[theorem]{Proposition}

\DeclareMathOperator{\Ind}{Ind}

\makeatletter
\def\moverlay{\mathpalette\mov@rlay}
\def\mov@rlay#1#2{\leavevmode\vtop{%
   \baselineskip\z@skip \lineskiplimit-\maxdimen
   \ialign{\hfil$\m@th#1##$\hfil\cr#2\crcr}}}
\newcommand{\charfusion}[3][\mathord]{
    #1{\ifx#1\mathop\vphantom{#2}\fi
        \mathpalette\mov@rlay{#2\cr#3}
      }
    \ifx#1\mathop\expandafter\displaylimits\fi}
\makeatother

\newcommand{\suchthat}{\;\ifnum\currentgrouptype=16 \middle\fi|\;}
\newcommand{\bR}{{\mathbb{R}}}      

\theoremstyle{definition}
\newtheorem{remark}[theorem]{Remark}
\newtheorem{definition}[theorem]{Definition}
\newtheorem{example}[theorem]{Example}

\usepackage{orcidlink}

\title[A nonlinear relapse model with disaggregated contact rates]{A nonlinear relapse model with disaggregated contact rates: analysis of a forward-backward bifurcation}
\author[J. Calvo-Monge]{Jimmy Calvo-Monge \orcidlink{0000-0002-4823-2455}}
\address{Escuela de Matem\'atica, Universidad de Costa Rica, San Jose, Costa Rica}
\email{jimmy.calvo@ucr.ac.cr}

\author[F. Sanchez]{Fabio Sanchez \orcidlink{0000-0002-5552-3672}}
\address{Centro de Investigaci\'on en Matem\'atica Pura y Aplicada-Escuela de Matem\'atica, Universidad de Costa Rica, San Jose, Costa Rica}
\email{fabio.sanchez@ucr.ac.cr}

\author[J.G. Calvo]{Juan G. Calvo \orcidlink{0000-0001-9948-9966}}
\address{Centro de Investigaci\'on en Matem\'atica Pura y Aplicada-Escuela de Matem\'atica, Universidad de Costa Rica, San Jose, Costa Rica}
\email{juan.calvo@ucr.ac.cr}

\author[D. Mena]{Dar\'io Mena \orcidlink{0000-0002-9443-391X}}
\address{Centro de Investigaci\'on en Matem\'atica Pura y Aplicada-Escuela de Matem\'atica, Universidad de Costa Rica, San Jose, Costa Rica}
\email{dario.menaarias@ucr.ac.cr}

\begin{document}

\maketitle

\begin{abstract}
Throughout the progress of epidemic scenarios, individuals in different health classes are expected to have different average daily contact behavior. This contact heterogeneity has been studied in recent adaptive models and allows us to capture the inherent differences across health statuses better. Diseases with reinfection bring out more complex scenarios and offer an important application to consider contact disaggregation. Therefore, we developed a nonlinear differential equation model to explore the dynamics of relapse phenomena and contact differences across health statuses. Our incidence rate function is formulated, taking inspiration from recent adaptive algorithms. It incorporates contact behavior for individuals in each health class. We use constant contact rates at each health status for our analytical results and prove conditions for different forward-backward bifurcation scenarios. The relationship between the different contact rates heavily influences these conditions. Numerical examples highlight the effect of temporarily recovered individuals and initial conditions on infected population persistence.
\end{abstract}




\section{Introduction}

Epidemiological models serve as an essential tool for understanding disease dynamics. Many historical examples yield insightful results on how initial conditions and parameters alter the progression of an epidemic outbreak \cite{Brur09}; critical concepts developed in this setting, such as the $R_0$ reproductive number, work as threshold indicators for disease behavior. Modern epidemiological mathematics heavily use bread-and-butter SIR models \cite{AndMay79}, and current research efforts in this area are devoted to modifying the classical models, allowing them to capture all the intricacies of real-world disease dynamics, for example, better representation of social distancing phenomena, compliance conducts, economic conditions and other factors.

One effort in this area is related to studying human contact behavior. Contacts between individuals of different characteristics (health statuses, age groups ...) constitute a key factor in disease spread \cite{Zhang20, Mossong08, Lauro21}. The need to study contact differences due to health status requires that classical models be modified. In classical settings, there is an implicit assumption of homogeneous behavior in each compartment (for example, among susceptible and infected individuals) through establishing constant or proportional contact rates. This approach hides different individuals' inherent characteristics and responses toward the disease's progress. 

We now have a history of multiple efforts to deal with this problem. A first approach consists of specifying non-linear incidence rate functions by constructing functions that reflect the impact of the state of the model on the contact rates through time; for example, \cite{Li17, Liu87,Xiao07,HU201212,Hethcote91,Alex06} for models without relapse, and \cite{Arr22,Xiao10} for models with non-linear relapse rates. In these cases, the general idea is to include functions of the form
\begin{align*}
    g_{\kappa,\nu}(\cdot)I = \frac{\kappa I^p}{1+\nu I^q},
\end{align*}
as the incidence rate function for the disease, using positive constants $\kappa,\nu,p,q$ and with $I = I(t)$ the infected population size in time. Within relapse phenomena, very similarly, \cite{Arr22} proposes the function
\begin{align}\label{c_arr}
    g_{\kappa,\nu}(\cdot) = \frac{\kappa}{1+ \nu \frac{R}{N}},
\end{align}
where $R= R(t)$ is the number of (temporarily) recovered individuals at time $t$, and $N$ is the total population size. As we can see, with this approach, modelers usually specify functions that decrease when the epidemic burden is high. This makes them depend inversely on the sizes of infected or recovered populations in time. The subsequent analysis is commonly focused on the impact of the model hyper-parameters (constants such as $\kappa,\nu$ in the non-linear functions) on the system behavior.

Key analytical results can be obtained using this approach. They can tackle the problem of different behaviors among health classes: what we call the \textit{epidemiological heterogeneity} of agents involved in the disease progression. As mentioned in \cite{Funk10}, although these models are rich in dynamic and analytical properties, the exact contact dynamic behavior sometimes is not emphasized in their formulation. In recent years new interest has been placed in representing more dynamic information on contact rates from the study population. Particularly, there has been interest in the \textit{economical heterogeneity} of individuals involved in epidemics and in the inclusion of utilitarian adaptive decisions individuals make within the development of the epidemic scenario.

The main contribution from \cite{Feni11} consists of devising a process in which contact rates for individuals in each health class can be updated simultaneously as disease changes. The idea consists of modeling the individuals as decision agents who consider their environment status and utility to decide optimal contact rates throughout time. This approach considers individuals' economic considerations when deciding how many contacts they should engage with at each time period. A detailed review of other proposals under the epi-economical approach is available in \cite{Funk10}.

This recent technique allows the computation of contact rates alongside the progress of the disease through an optimization decision process performed by each individual. We call this procedure the \textit{adaptive setting}. This has proved helpful in creating epidemiological models closer to the actual decision-making processes made by individuals. It has been applied to create more realistic settings and compare them with the classical formulation. For example, thanks to the use of the adaptive setting, there are novel insights on the true impact of asymptomatic individuals \cite{Balt2}, a more intuitive understanding of final epidemic burden states in contrast to the classical results \cite{Balt1}, and a deeper analysis on social distancing \cite{Feni11_2}. Analytical comparisons and conjectures for the adaptive setting can be found under the non-relapse case in \cite{CCast13}.

The adaptive setting is not detached from the first approach. To compute contact rates adaptively, we must first define non-linear incidence rate functions that will use these contact rates. The formulation of non-linear incidence rate functions in the adaptive setting is commonly expressed in the form:
\begin{align}\label{C_orig}
    g(S,I,R)= \frac{C^sC^iN}{SC^s + IC^i + R C^{r}},
\end{align}
where $C^h=C^h(S,I,R)$ (for $h\in\{s,i,r\})$ is the average number of contacts for each health status individual per time period, and $N$ is the total population size. These contact rates might be functions that depend on the status of the disease $S,I,R$, especially when using the adaptive approach, where they are updated throughout the disease dynamics.

This adaptive setting constitutes a recent effort and offers a promising strategy to capture complex epi-economical phenomena better. Given its novelty, the literature on adaptive behavior has not been applied to non-linear relapse scenarios. Although several references propose non-linear relapse incidence rate functions for epidemiological differential equation models, the formulation (\ref{C_orig}) merits further analytical inspection in the relapse scenario. This paper uses this formula for incidence rate functions to study a relapse model. We will examine the analytical impact of specifying contact rates using (\ref{C_orig}) and the repercussions on how to interpret these models. Our results will be framed in terms of the relations between the contact rates $C^s, C^i$, and $C^{r}$ when they are assumed constant. In Section \ref{section2}, we propose our model and explore its main analytical properties. We present our main theoretical results in Section \ref{section3}, where bifurcation plots and local stability are considered; all mathematical proofs can be found in Section \ref{section6}. Section \ref{section4} provides some numerical simulations of sensitivity to contact rates and initial conditions. A discussion of our main results can be found in Section \ref{section6}.

\section{Non-linear relapse rate model}\label{section2}

We propose an epidemiological model with the presence of non-linear relapse behavior. Following \cite{Schz19} and \cite{Snchz07}, we consider three compartments of individuals: $S$ (susceptible), $I$ (infected), and $R$ (recovered with the possibility of reinfection) and represent the model dynamics using the following system of equations,
\begin{align}\label{model1}
    & \frac{dS}{dt}=  -g(\cdot)\beta \frac{SI}{N} + \mu N - \mu S, \nonumber\\
    & \frac{dI}{dt}= g(\cdot)\beta \frac{SI}{N} + \phi\frac{R I}{N}-(\gamma+\mu)I,  \\
    & \frac{dR}{dt}= \gamma I - \phi \frac{IR}{N} - \mu R.\nonumber
\end{align}

$N = S + I + R$ is constant. The function $g(\cdot)$ is the incidence rate function of the model. From now on, we will take $g(\cdot)$ given by (\ref{C_orig}). This is the proposed relaxation of the burden of health status homogeneity, present in the classical formulation. The likelihood of infection when there is contact with an infected individual is given by $\beta$, the rate of recovery by $\gamma$, and the rate of reinfection represented by $\phi$. We have a demographic exit and entrance rate for the system given by $\mu$. The incidence rate function, $g(\cdot)$, represents the contact rate between susceptible and infected individuals, implying that $g(\cdot)\beta$ acts as the rate at which susceptible become infected.

We re-scale the system (\ref{model1}) by substituting $s := \frac{S}{N}, i := \frac{I}{N}$ and $r:= \frac{R}{N}$, to obtain the equivalent model
\begin{subequations}\label{model}
\begin{align}
    & \frac{ds}{dt}=  -g(\cdot)\beta si + \mu  - \mu s. \label{modeleqn1} \\
    & \frac{di}{dt}= g(\cdot)\beta si + \phi r i-(\gamma+\mu)i.  \label{modeleqn2}  \\
    & \frac{dr}{dt}= \gamma i - \phi r i - \mu r.  \label{modeleqn3} 
\end{align}
\end{subequations}

The incidence rate function $g(\cdot)$ can also be re-scaled and substituted by: $$g(\cdot) = g(s,i,r) = \frac{C^sC^i}{sC^s + iC^i + r C^{r}}.$$

\begin{remark}\label{C_functions_remark}
In general, contact rates are functions that depend on the status of the disease, that is $C^h= C^h(S,I,R)$ for each $h \in \{s,i,r\}$. For the remainder of this article, we consider these functions constant. We aim to generalize mathematical results obtained in \mbox{\cite{Arr22}} and elucidate possible analytical properties of the adaptive algorithm in the relapse case. In this case, our mathematical analysis, including the calculation of $R_0$ and the determination of stable equilibria, will be greatly simplified. As will be seen shortly, all our results are greatly influenced by the ratios between the contact rates $C^h$.
\end{remark}

\section{Mathematical Analysis}\label{section3}

\subsection{Basic Reproductive Number}

Using the next generation matrix approach \cite{Hethcote00}, we compute a basic reproductive number $R_0$ for this system. Here, it is simple to see that 
\begin{align*}
R_0 = \frac{\beta}{\gamma+\mu} \lim_{(s,i,r) \to (1,0,0)} g(s,i,r) = \frac{\beta}{\gamma+\mu}C_0.
\end{align*} 
Thus, $R_0$ depends on $C_0$, the limit of the incidence function value when the system converges to the disease-free state. When all contact coefficients $C^h$ are constant, then $C_0 = C^i$.

\subsection{Finding equilibrium points}

First, we study the disease-free equilibrium, where $(s(t),i(t),r(t)) = (1,0,0)$. 

\begin{theorem}
    The disease-free equilibrium is stable if and only if $R_0 < 1$.
\end{theorem}

\begin{proof}
    Note that the Jacobian matrix of the system (\ref{model}) is given by
    \begin{align*}
        J(s,i,r) = \begin{pmatrix}
        -\beta i ( g_s s + g) - \mu & -\beta s (g_i i + g) & -\beta s i g_{r} \\
        \beta i ( g_s s + g) & \beta s (g_i i + g) + \phi r - (\mu + \gamma) & \beta s i g_{r} + \phi i \\
        0 & \gamma - \phi r & - \phi i - \mu
        \end{pmatrix},
    \end{align*}
    where $g_h$ is the partial derivative of $g$ with respect to the variable $h \in \{s,i,r \}$. Taking the limit to the disease-free point, we get
    \begin{align*}
        \lim_{(s,i,r) \to (1,0,0)} J(s,i,r) = \begin{pmatrix}
        - \mu & -\beta C_0 & 0 \\
        0 & \beta C_0 - (\mu + \gamma) & 0 \\
        0 & \gamma & - \mu
        \end{pmatrix},
    \end{align*}
    which has eigenvalues $\lambda_1, \lambda_2 = - \mu$ and $\lambda_3 = \beta C_0 - (\mu + \gamma)$. This point is stable if and only if $\lambda_3<0$, which is equivalent to $R_0<1$.
 \end{proof}

The case of epidemiological interest is when $R_0>1$, in which we study the existence of endemic equilibria. Initial calculations show that these points must be in the form$
\left(1-i^*-\frac{\gamma i^*}{\phi i^*+ \mu}, i^*, \frac{\gamma i^*}{\phi i^*+ \mu}\right)$. To find the value of $i^*$ at the equilibrium, we can substitute this point into (\ref{modeleqn2}), use that $r = 1 - s - i$ (as $N$ is constant in this model) and obtain that $i^*$ must satisfy the cubic equation $a_3X^3+ a_2 X^2 + a_1X + a_0 = 0$, whose coefficients are:
\begin{align}\label{coefficients}
    & a_3 = R_\phi^2 R_0 - R_\mu R_\phi^2\left(1- \kappa\right), \nonumber \\
    & a_2 = R_\phi \biggl[R_0(1-R_\phi) +R_\mu( R_0 + R_\phi) - R_\mu(1-R_\mu)\left(1- \theta\right) -R_\mu (1+R_\mu)\left(1 - \kappa\right) \biggr],  \nonumber \\
    & a_1 = R_\mu \biggl[ R_0(1-R_\phi) +  R_\phi(1-R_0) - (1-R_\mu)\left(1 - \theta\right) + R_\mu R_\phi - R_\mu \left(1 - \kappa\right) \biggr],  \nonumber \\
    & a_0 = R_\mu^2(1-R_0),
\end{align}
where 
\begin{align}\label{rphirmu}
    \kappa = \frac{C^i}{C^s},  \quad
 \theta = \frac{C^{r}}{C^s}, \quad R_{\mu}=\frac{\mu}{\mu+\gamma}, \quad R_\phi = \frac{\phi}{\mu+\gamma}.
\end{align}
Mathematically, the model proposed in \cite{Arr22} can be seen as a special case of our model: if we use $C^s=C^i$ and $C^{r}=C^i(1+\nu)$, we obtain $g(\cdot)$ given by (\ref{c_arr}). This cubic equation also becomes the generalization of the corresponding one obtained in \cite{Arr22}. We also point out the biological interpretation of these contact quotients: $\kappa$ represents the change expected in contacts made by an individual after it becomes infected, and $\theta$ compares the difference between the individual contacts before infection and after recovery.

We proceed to examine the behavior and existence of equilibria points based only on the disease parameters of the model ($R_\phi$, $R_\mu$), the infected individual response to the disease ($R_0$), and the relationship between the average contact rates between compartments ($\frac{C^i}{C^s}$ and $\frac{C^{r}}{C^s}$). In Figure (\ref{fig_cubic}) we use  $\mu = 0.00015, \gamma = 0.0027, \beta = 0.00096$ and $\phi = 0.044$, taken from simulations made in \cite{Arr22} and drug epidemic parameter estimation performed in \cite{Snchz07}. We create bifurcation plots for each $R_0$ and varying the quotients $\kappa = \frac{C^i}{C^s}$ and $\theta = \frac{C^{r}}{C^s}$. First, see Figure (\ref{fig_cubic}) for $\theta = 1.7$.

\begin{figure}[H]

    \begin{subfigure}[b]{0.45\textwidth}
         \begin{tikzpicture}[scale=0.8]
            \begin{axis}[
                title = $\kappa \text{ = } 0.8 \text{, } \theta \text{ = } 1.7$ ,
                axis lines=middle,
                cycle list name=black white,
                xmin=0.8,xmax=1.15,ymin=0,ymax=0.3,
                ytick={0.05,0.1,0.15,0.2,0.25},
                xlabel={$R_0(C^i)$},
                ylabel={$i^*$},
                ylabel near ticks,
                xlabel near ticks,
            ]
            \addplot+[
                blue,
                only marks,
                mark size=1.25pt]
            table[x=x, y=y]
            {Simulations_Data/dat0_stable.dat};

            \draw [line width=2.5pt, orange, dashed] (1.05,0.15) to (1.1,0.15);
            \node[below] at (1.075,0.15){$\text{\tiny{Unstable}}$};
            
            \draw [line width=2.5pt, blue] (1.05,0.1) to (1.1,0.1);
            \node[below] at (1.075,0.1){$\text{\tiny{Stable}}$};

            \draw [line width=2.5pt, orange] (0.854504505,0.12054532) to (0.857777778,0.109138126);

            \draw [line width=2.5pt, orange] (0.861051051,0.102657229) to (0.864324324,0.09761979);
        
            \draw [line width=2.5pt, orange] (0.867597598,0.09336489) to (0.870870871,0.089620874);
        
            \draw [line width=2.5pt, orange] (0.874144144,0.086244302) to (0.877417417,0.083148373);
        
            \draw [line width=2.5pt, orange] (0.880690691,0.080275764) to (0.883963964,0.077586228);
        
            \draw [line width=2.5pt, orange] (0.887237237,0.075050189) to (0.890510511,0.072645129);
        
            \draw [line width=2.5pt, orange] (0.893783784,0.070353413) to (0.897057057,0.0681609);
        
            \draw [line width=2.5pt, orange] (0.90033033,0.066056029) to (0.903603604,0.064029184);
        
            \draw [line width=2.5pt, orange] (0.906876877,0.062072247) to (0.91015015,0.060178278);
        
            \draw [line width=2.5pt, orange] (0.913423423,0.058341266) to (0.916696697,0.056555951);
        
            \draw [line width=2.5pt, orange] (0.91996997,0.054817681) to (0.923243243,0.053122297);
        
            \draw [line width=2.5pt, orange] (0.926516517,0.05146605) to (0.92978979,0.04984552);
        
            \draw [line width=2.5pt, orange] (0.933063063,0.048257561) to (0.936336336,0.046699248);
        
            \draw [line width=2.5pt, orange] (0.93960961,0.045167829) to (0.942882883,0.043660687);
        
            \draw [line width=2.5pt, orange] (0.946156156,0.042175304) to (0.949429429,0.040709221);
        
            \draw [line width=2.5pt, orange] (0.952702703,0.039260004) to (0.955975976,0.037825203);
        
            \draw [line width=2.5pt, orange] (0.959249249,0.036402314) to (0.962522523,0.034988726);
        
            \draw [line width=2.5pt, orange] (0.965795796,0.033581661) to (0.969069069,0.032178108);
        
            \draw [line width=2.5pt, orange] (0.972342342,0.030774715) to (0.975615616,0.029367671);
        
            \draw [line width=2.5pt, orange] (0.978888889,0.027952517) to (0.982162162,0.026523888);
        
            \draw [line width=2.5pt, orange] (0.985435435,0.025075117) to (0.988708709,0.023597611);
        
            \draw [line width=2.5pt, orange] (0.991981982,0.022079805) to (0.995255255,0.020505288);
        
            \draw [line width=2.5pt, orange] (0.998528529,0.018849108) to (1.001801802,0.017069502);
        
            \draw [line width=2.5pt, orange] (1.005075075,0.015085312) to (1.008348348,0.012687897);
            
            \addplot +[red, dashed, mark=none] coordinates {(0.854504505, -0.05) (0.854504505, 0.27)};
            \addplot +[red, dashed, mark=none] coordinates {(1, -0.05) (1, 0.27)};
            \addplot +[red, dashed,mark=none] coordinates {(1.0116, -0.05) (1.0116, 0.27)};
            \node at (0.83,0.28){$\mathcal{R}_1$};
            \node at (0.925,0.28){$\mathcal{R}_2$};
            \node at (1.01,0.28){$\mathcal{R}_3$};
            \node at (1.075,0.28){$\mathcal{R}_4$};
            \end{axis}
        \end{tikzpicture}
     \end{subfigure}\hspace{5mm}
    \begin{subfigure}[b]{0.45\textwidth}
         \begin{tikzpicture}[scale=0.8]
            \begin{axis}[
                title = $\kappa \text{ = } 0.5 \text{, } \theta \text{ = } 1.7$ ,
                axis lines=middle,
                cycle list name=black white,
                xmin=0.8,xmax=1.12,ymin=0,ymax=0.32,
                ytick={0.05,0.1,0.15,0.2,0.25},
                xlabel={$R_0(C^i)$},
                ylabel={$i^*$},
                ylabel near ticks,
                xlabel near ticks,
            ]
            \addplot+[
                blue,
                only marks,
                mark size=1.25pt]
            table[x=x, y=y]
            {Simulations_Data/dat1_stable.dat};

            \draw [line width=2.5pt, orange, dashed] (1.05,0.15) to (1.1,0.15);
            \node[below] at (1.075,0.15){$\text{\tiny{Unstable}}$};
            
            \draw [line width=2.5pt, blue] (1.05,0.1) to (1.1,0.1);
            \node[below] at (1.075,0.1){$\text{\tiny{Stable}}$};

            \draw [line width=2.5pt, orange] (0.829409409,0.12616413) to (0.832682683,0.116013695);

            \draw [line width=2.5pt, orange] (0.835955956,0.10965553) to (0.839229229,0.104629799);
        
            \draw [line width=2.5pt, orange] (0.842502503,0.100354012) to (0.845775776,0.09657688);
        
            \draw [line width=2.5pt, orange] (0.849049049,0.093162494) to (0.852322322,0.090027399);
        
            \draw [line width=2.5pt, orange] (0.855595596,0.087115941) to (0.858868869,0.084388768);
        
            \draw [line width=2.5pt, orange] (0.862142142,0.081816814) to (0.865415415,0.079377868);
        
            \draw [line width=2.5pt, orange] (0.868688689,0.077054486) to (0.871961962,0.074832658);
        
            \draw [line width=2.5pt, orange] (0.875235235,0.072700918) to (0.878508509,0.07064973);
        
            \draw [line width=2.5pt, orange] (0.881781782,0.06867105) to (0.885055055,0.06675801);
        
            \draw [line width=2.5pt, orange] (0.888328328,0.064904679) to (0.891601602,0.063105889);
        
            \draw [line width=2.5pt, orange] (0.894874875,0.06135709) to (0.898148148,0.059654245);
        
            \draw [line width=2.5pt, orange] (0.901421421,0.057993744) to (0.904694695,0.056372332);
        
            \draw [line width=2.5pt, orange] (0.907967968,0.054787055) to (0.911241241,0.05323521);
        
            \draw [line width=2.5pt, orange] (0.914514515,0.051714307) to (0.917787788,0.050222036);
        
            \draw [line width=2.5pt, orange] (0.921061061,0.048756238) to (0.924334334,0.047314876);
        
            \draw [line width=2.5pt, orange] (0.927607608,0.045896017) to (0.930880881,0.044497809);
        
            \draw [line width=2.5pt, orange] (0.934154154,0.043118457) to (0.937427427,0.041756213);
        
            \draw [line width=2.5pt, orange] (0.940700701,0.040409346) to (0.943973974,0.03907613);
        
            \draw [line width=2.5pt, orange] (0.947247247,0.037754817) to (0.950520521,0.036443619);
        
            \draw [line width=2.5pt, orange] (0.953793794,0.035140674) to (0.957067067,0.033844017);
        
            \draw [line width=2.5pt, orange] (0.96034034,0.032551538) to (0.963613614,0.031260933);
        
            \draw [line width=2.5pt, orange] (0.966886887,0.029969638) to (0.97016016,0.028674742);
        
            \draw [line width=2.5pt, orange] (0.973433433,0.027372866) to (0.976706707,0.026059999);
        
            \draw [line width=2.5pt, orange] (0.97997998,0.024731251) to (0.983253253,0.023380485);
        
            \draw [line width=2.5pt, orange] (0.986526527,0.021999733) to (0.9897998,0.020578222);
        
            \draw [line width=2.5pt, orange] (0.993073073,0.01910063) to (0.996346346,0.017543676);
        
            \draw [line width=2.5pt, orange] (0.99961962,0.015868527) to (1.002892893,0.014000362);
            
            \addplot +[red, dashed, mark=none] coordinates {(0.829409409, -0.05) (0.829409409, 0.27)};
            \addplot +[red, dashed, mark=none] coordinates {(1, -0.05) (1, 0.27)};
            \addplot +[red, dashed,mark=none] coordinates {(1.0116, -0.05) (1.0116, 0.27)};
            \node at (0.83,0.29){$\mathcal{R}_1$};
            \node at (0.925,0.29){$\mathcal{R}_2$};
            \node at (1.01,0.29){$\mathcal{R}_3$};
            \node at (1.075,0.29){$\mathcal{R}_4$};
            \end{axis}
        \end{tikzpicture}
     \end{subfigure}
    \begin{subfigure}[b]{0.45\textwidth}
         \begin{tikzpicture}[scale=0.8]
            \begin{axis}[
                title = $\kappa \text{ = } 0.3 \text{, } \theta \text{ = } 1.7$ ,
                axis lines=middle,
                cycle list name=black white,
                xmin=0.75,xmax=1.12,ymin=0,ymax=0.32,
                ytick={0.05,0.1,0.15,0.2,0.25},
                xlabel={$R_0(C^i)$},
                ylabel={$i^*$},
                ylabel near ticks,
                xlabel near ticks,
            ]
            \addplot+[
                blue,
                only marks,
                mark size=1.25pt]
            table[x=x, y=y]
            {Simulations_Data/dat2_stable.dat};

            \draw [line width=2.5pt, orange, dashed] (1.05,0.15) to (1.1,0.15);
            \node[below] at (1.075,0.15){$\text{\tiny{Unstable}}$};
            
            \draw [line width=2.5pt, blue] (1.05,0.1) to (1.1,0.1);
            \node[below] at (1.075,0.1){$\text{\tiny{Stable}}$};

            \draw [line width=2.5pt, orange] (0.810860861,0.135485902) to (0.814134134,0.123055826);

            \draw [line width=2.5pt, orange] (0.817407407,0.116210351) to (0.820680681,0.1109109);
        
            \draw [line width=2.5pt, orange] (0.823953954,0.106443256) to (0.827227227,0.102517298);
        
            \draw [line width=2.5pt, orange] (0.830500501,0.098980695) to (0.833773774,0.095741585);
        
            \draw [line width=2.5pt, orange] (0.837047047,0.092739456) to (0.84032032,0.089931904);
        
            \draw [line width=2.5pt, orange] (0.843593594,0.087287828) to (0.846866867,0.08478359);
        
            \draw [line width=2.5pt, orange] (0.85014014,0.082400716) to (0.853413413,0.080124426);
        
            \draw [line width=2.5pt, orange] (0.856686687,0.077942668) to (0.85995996,0.075845451);
        
            \draw [line width=2.5pt, orange] (0.863233233,0.073824376) to (0.866506507,0.071872292);
        
            \draw [line width=2.5pt, orange] (0.86977978,0.069983049) to (0.873053053,0.0681513);
        
            \draw [line width=2.5pt, orange] (0.876326326,0.06637236) to (0.8795996,0.064642087);
        
            \draw [line width=2.5pt, orange] (0.882872873,0.062956795) to (0.886146146,0.061313177);
        
            \draw [line width=2.5pt, orange] (0.889419419,0.05970825) to (0.892692693,0.058139304);
        
            \draw [line width=2.5pt, orange] (0.895965966,0.056603863) to (0.899239239,0.055099651);
        
            \draw [line width=2.5pt, orange] (0.902512513,0.053624562) to (0.905785786,0.052176639);
        
            \draw [line width=2.5pt, orange] (0.909059059,0.050754048) to (0.912332332,0.049355063);
        
            \draw [line width=2.5pt, orange] (0.915605606,0.047978049) to (0.918878879,0.046621443);
        
            \draw [line width=2.5pt, orange] (0.922152152,0.045283745) to (0.925425425,0.043963501);
        
            \draw [line width=2.5pt, orange] (0.928698699,0.042659291) to (0.931971972,0.041369716);
        
            \draw [line width=2.5pt, orange] (0.935245245,0.040093388) to (0.938518519,0.038828909);
        
            \draw [line width=2.5pt, orange] (0.941791792,0.037574863) to (0.945065065,0.036329792);
        
            \draw [line width=2.5pt, orange] (0.948338338,0.035092185) to (0.951611612,0.033860446);
        
            \draw [line width=2.5pt, orange] (0.954884885,0.032632871) to (0.958158158,0.031407612);
        
            \draw [line width=2.5pt, orange] (0.961431431,0.030182633) to (0.964704705,0.028955652);
        
            \draw [line width=2.5pt, orange] (0.967977978,0.027724064) to (0.971251251,0.026484835);
        
            \draw [line width=2.5pt, orange] (0.974524525,0.025234362) to (0.977797798,0.023968254);
        
            \draw [line width=2.5pt, orange] (0.981071071,0.022681018) to (0.984344344,0.021365566);
        
            \draw [line width=2.5pt, orange] (0.987617618,0.020012405) to (0.990890891,0.018608213);
        
            \draw [line width=2.5pt, orange] (0.994164164,0.017133143) to (0.997437437,0.015555073);
        
            \draw [line width=2.5pt, orange] (1.000710711,0.013815106) to (1.003983984,0.011779119);
            
            \addplot +[red, dashed, mark=none] coordinates {(0.810860861, -0.05) (0.810860861, 0.29)};
            \addplot +[red, dashed, mark=none] coordinates {(1, -0.05) (1, 0.29)};
            \addplot +[red, dashed,mark=none] coordinates {(1.0116, -0.05) (1.0116, 0.29)};
            \node at (0.79,0.3){$\mathcal{R}_1$};
            \node at (0.91,0.3){$\mathcal{R}_2$};
            \node at (1.01,0.3){$\mathcal{R}_3$};
            \node at (1.075,0.3){$\mathcal{R}_4$};
            \end{axis}
        \end{tikzpicture}
     \end{subfigure}\hspace{5mm}
    \begin{subfigure}[b]{0.45\textwidth}
         \begin{tikzpicture}[scale=0.8]
            \begin{axis}[
                title = $\kappa \text{ = } 0.01 \text{, } \theta \text{ = } 1.7$ ,
                axis lines=middle,
                cycle list name=black white,
                xmin=0.75,xmax=1.12,ymin=0,ymax=0.33,
                ytick={0.05,0.1,0.15,0.2,0.25},
                xlabel={$R_0(C^i)$},
                ylabel={$i^*$},
                ylabel near ticks,
                xlabel near ticks,
            ]
            \addplot+[
                blue,
                only marks,
                mark size=1.25pt]
            table[x=x, y=y]
            {Simulations_Data/dat3_stable.dat};

            \draw [line width=2.5pt, orange, dashed] (1.05,0.15) to (1.1,0.15);
            \node[below] at (1.075,0.15){$\text{\tiny{Unstable}}$};
            
            \draw [line width=2.5pt, blue] (1.05,0.1) to (1.1,0.1);
            \node[below] at (1.075,0.1){$\text{\tiny{Stable}}$};

            \draw [line width=2.5pt, orange] (0.782492492,0.143826171) to (0.785765766,0.131858265);
        
            \draw [line width=2.5pt, orange] (0.789039039,0.124919609) to (0.792312312,0.11950808);
        
            \draw [line width=2.5pt, orange] (0.795585586,0.114931225) to (0.798858859,0.110902144);
        
            \draw [line width=2.5pt, orange] (0.802132132,0.10726874) to (0.805405405,0.103938753);
        
            \draw [line width=2.5pt, orange] (0.808678679,0.100851167) to (0.811951952,0.097963105);
        
            \draw [line width=2.5pt, orange] (0.815225225,0.09524305) to (0.818498498,0.09266701);
        
            \draw [line width=2.5pt, orange] (0.821771772,0.090216213) to (0.825045045,0.087875626);
        
            \draw [line width=2.5pt, orange] (0.828318318,0.085632985) to (0.831591592,0.08347812);
        
            \draw [line width=2.5pt, orange] (0.834864865,0.081402481) to (0.838138138,0.079398793);
        
            \draw [line width=2.5pt, orange] (0.841411411,0.077460799) to (0.844684685,0.075583068);
        
            \draw [line width=2.5pt, orange] (0.847957958,0.073760846) to (0.851231231,0.071989937);
        
            \draw [line width=2.5pt, orange] (0.854504505,0.070266614) to (0.857777778,0.068587547);
        
            \draw [line width=2.5pt, orange] (0.861051051,0.066949738) to (0.864324324,0.065350476);
        
            \draw [line width=2.5pt, orange] (0.867597598,0.063787298) to (0.870870871,0.062257951);
        
            \draw [line width=2.5pt, orange] (0.874144144,0.060760368) to (0.877417417,0.059292642);
        
            \draw [line width=2.5pt, orange] (0.880690691,0.057853005) to (0.883963964,0.056439813);
        
            \draw [line width=2.5pt, orange] (0.887237237,0.055051527) to (0.890510511,0.053686704);
        
            \draw [line width=2.5pt, orange] (0.893783784,0.052343981) to (0.897057057,0.051022066);
        
            \draw [line width=2.5pt, orange] (0.90033033,0.049719728) to (0.903603604,0.048435791);
        
            \draw [line width=2.5pt, orange] (0.906876877,0.04716912) to (0.91015015,0.04591862);
        
            \draw [line width=2.5pt, orange] (0.913423423,0.044683222) to (0.916696697,0.043461882);
        
            \draw [line width=2.5pt, orange] (0.91996997,0.042253568) to (0.923243243,0.041057256);
        
            \draw [line width=2.5pt, orange] (0.926516517,0.039871921) to (0.92978979,0.03869653);
        
            \draw [line width=2.5pt, orange] (0.933063063,0.03753003) to (0.936336336,0.03637134);
        
            \draw [line width=2.5pt, orange] (0.93960961,0.035219338) to (0.942882883,0.034072849);
        
            \draw [line width=2.5pt, orange] (0.946156156,0.032930626) to (0.949429429,0.031791331);
        
            \draw [line width=2.5pt, orange] (0.952702703,0.030653511) to (0.955975976,0.029515567);
        
            \draw [line width=2.5pt, orange] (0.959249249,0.028375712) to (0.962522523,0.027231923);
        
            \draw [line width=2.5pt, orange] (0.965795796,0.026081866) to (0.969069069,0.024922807);
        
            \draw [line width=2.5pt, orange] (0.972342342,0.023751473) to (0.975615616,0.022563868);
        
            \draw [line width=2.5pt, orange] (0.978888889,0.021354982) to (0.982162162,0.020118352);
        
            \draw [line width=2.5pt, orange] (0.985435435,0.018845346) to (0.988708709,0.017523913);
        
            \draw [line width=2.5pt, orange] (0.991981982,0.016136262) to (0.995255255,0.014653997);
        
            \draw [line width=2.5pt, orange] (0.998528529,0.013026272) to (1.001801802,0.011142452);
            
            \addplot +[red, dashed, mark=none] coordinates {(0.782492492, -0.05) (0.782492492, 0.31)};
            \addplot +[red, dashed, mark=none] coordinates {(1, -0.05) (1, 0.31)};
            \addplot +[red, dashed,mark=none] coordinates {(1.0116, -0.05) (1.0116, 0.31)};
            \node at (0.77,0.32){$\mathcal{R}_1$};
            \node at (0.9,0.32){$\mathcal{R}_2$};
            \node at (1.01,0.32){$\mathcal{R}_3$};
            \node at (1.075,0.32){$\mathcal{R}_4$};
            \end{axis}
        \end{tikzpicture}
     \end{subfigure}

    \caption{Equilibria points computed using $\theta = 1.7$ and varying $\kappa = 0.8, 0.5, 0.3 $ and $0.01$. A cubic bifurcation plot can be found, and three equilibria points occur within an interval $R_0 \in [1,1+\epsilon(\kappa, \theta)]$. We note that decreasing the value of $\kappa$ diminishes the window $\epsilon(\kappa,\theta)$, and it decreases the minimal $R_0$ value for which we find stable equilibria, for example, for $\kappa=0.8$ this value is at $R_0 \approx 0.85$, but for $\kappa=0.01$ it is at $R_0 \approx 0.8$. For a discussion on the length of the window $\epsilon(\kappa,\theta)$, refer to Figure \mbox{\ref{R3_window_plot}} below. Stable and unstable equilibria regions are highlighted in these plots.}
    \label{fig_cubic}

\end{figure}
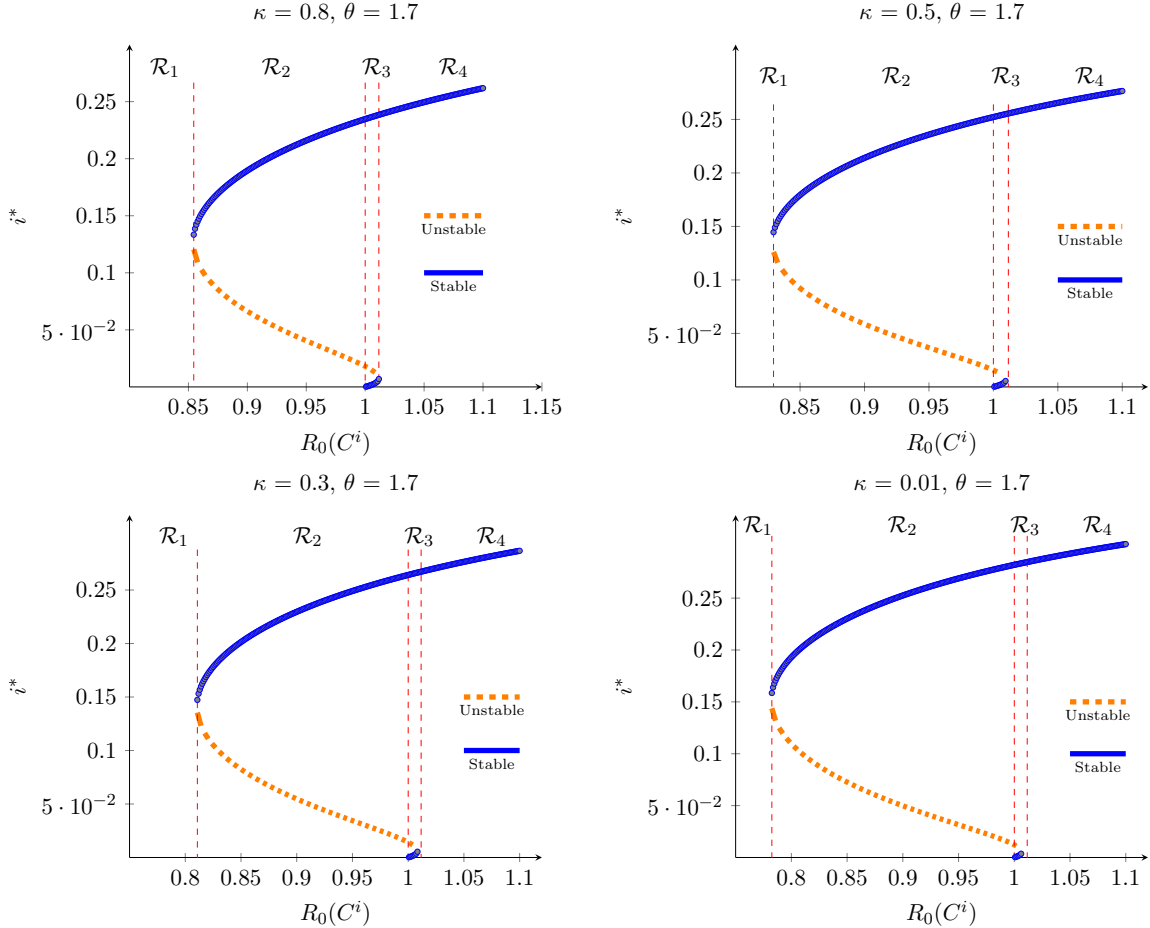

We divided these plots into four regions of interest for the basic reproductive number: $\mathcal{R}_1$ where no endemic equilibrium is attained, $\mathcal{R}_2$ where a stable endemic equilibrium, and another non-stable can be found, $\mathcal{R}_3$ where three possible equilibrium states can be found, one of which is stable, and $\mathcal{R}_4$ where there is just one stable, steady state.

\begin{remark}\label{importance_cubic}
The presence of this cubic phenomenon was first observed in \mbox{\cite{Arr22}}. This case is of particular interest when analyzing the effect of the value of $R_0$ in epidemics with relapse. The small region in the $[1, 1+ \epsilon(\kappa,\theta)]$ interval represents the possibility of having a very small stable equilibrium state of the disease even when the reproductive number is higher than $1$. These simulations suggest the importance of the contact rate $C^r$ in creating such a scenario.
\end{remark} 

On the other hand, decreasing the value of $\theta$ leads to a different behavior, as shown in Figure (\ref{fig_no_cubic}). In this case, we count three regions of interest exhibiting a typical backward quadratic bifurcation plot.

\begin{figure}[H]

    \begin{subfigure}[b]{0.45\textwidth}
         \begin{tikzpicture}[scale=0.7]
            \begin{axis}[
                title = $\kappa \text{ = } 0.8 \text{, } \theta \text{ = } 1.2$ ,
                axis lines=middle,
                cycle list name=black white,
                xmin=0.65,xmax=1.15,ymin=0,ymax=0.3,
                ytick={0.05,0.1,0.15,0.2,0.25},
                xlabel={$R_0(C^i)$},
                ylabel={$i^*$},
                ylabel near ticks,
                xlabel near ticks,
            ]
            \addplot+[
                blue,
                only marks,
                mark size=1.25pt]
            table[x=x, y=y]
            {Simulations_Data/dat4_estable.dat};

            \draw [line width=2.5pt, orange, dashed] (1.05,0.15) to (1.125,0.15);
            \node[below] at (1.0875,0.15){$\text{\tiny{Unstable}}$};
            
            \draw [line width=2.5pt, blue] (1.05,0.1) to (1.125,0.1);
            \node[below] at (1.0875,0.1){$\text{\tiny{Stable}}$};
            
            \draw [line width=2.5pt, orange] (0.696296296,0.129169612) to (0.707207207,0.104650083);

            \draw [line width=2.5pt, orange] (0.718118118,0.092424913) to (0.729029029,0.083264345);
        
            \draw [line width=2.5pt, orange] (0.73993994,0.075740845) to (0.750850851,0.06928163);
        
            \draw [line width=2.5pt, orange] (0.761761762,0.06358666) to (0.772672673,0.058474786);
        
            \draw [line width=2.5pt, orange] (0.783583584,0.053826331) to (0.794494494,0.049557116);
        
            \draw [line width=2.5pt, orange] (0.805405405,0.045605111) to (0.816316316,0.041922932);
        
            \draw [line width=2.5pt, orange] (0.827227227,0.038473318) to (0.838138138,0.035226259);
        
            \draw [line width=2.5pt, orange] (0.849049049,0.032157075) to (0.85995996,0.029245102);
        
            \draw [line width=2.5pt, orange] (0.870870871,0.026472742) to (0.881781782,0.023824769);
        
            \draw [line width=2.5pt, orange] (0.892692693,0.021287797) to (0.903603604,0.018849856);
        
            \draw [line width=2.5pt, orange] (0.914514515,0.016500052) to (0.925425425,0.014228252);
        
            \draw [line width=2.5pt, orange] (0.936336336,0.012024805) to (0.947247247,0.009880235);
        
            \draw [line width=2.5pt, orange] (0.958158158,0.007784903) to (0.969069069,0.005728553);
        
            \draw [line width=2.5pt, orange] (0.97997998,0.003699683) to (0.990890891,0.001684535);
            
            \addplot +[red, dashed, mark=none] coordinates {(0.696296296, -0.05) (0.696296296, 0.3)};
            \addplot +[red, dashed, mark=none] coordinates {(1, -0.05) (1, 0.3)};
            \node at (0.675,0.2){$\mathcal{R}_1$};
            \node at (0.85,0.2){$\mathcal{R}_2$};
            \node at (1.05,0.2){$\mathcal{R}_3$};
            \end{axis}
        \end{tikzpicture}
     \end{subfigure}\hspace{5mm}
    \begin{subfigure}[b]{0.45\textwidth}
         \begin{tikzpicture}[scale=0.7]
            \begin{axis}[
                title = $\kappa \text{ = } 0.5 \text{, } \theta \text{ = } 1.2$ ,
                axis lines=middle,
                cycle list name=black white,
                xmin=0.6,xmax=1.15,ymin=0,ymax=0.32,
                ytick={0.05,0.15,0.15,0.2,0.25},
                xlabel={$R_0(C^i)$},
                ylabel={$i^*$},
                ylabel near ticks,
                xlabel near ticks,
            ]
            \addplot+[
                blue,
                only marks,
                mark size=1.25pt]
            table[x=x, y=y]
            {Simulations_Data/dat5_estable.dat};

            \draw [line width=2.5pt, orange, dashed] (1.05,0.15) to (1.125,0.15);
            \node[below] at (1.0875,0.15){$\text{\tiny{Unstable}}$};
            
            \draw [line width=2.5pt, blue] (1.05,0.1) to (1.125,0.1);
            \node[below] at (1.0875,0.1){$\text{\tiny{Stable}}$};
            
            \draw [line width=2.5pt, orange] (0.667927928,0.144637192) to (0.678838839,0.114279328);

            \draw [line width=2.5pt, orange] (0.68974975,0.101226266) to (0.700660661,0.091515005);
        
            \draw [line width=2.5pt, orange] (0.711571572,0.083561706) to (0.722482482,0.076744748);
        
            \draw [line width=2.5pt, orange] (0.733393393,0.070741694) to (0.744304304,0.065358957);
        
            \draw [line width=2.5pt, orange] (0.755215215,0.060469108) to (0.766126126,0.055982746);
        
            \draw [line width=2.5pt, orange] (0.777037037,0.051834137) to (0.787947948,0.047973165);
        
            \draw [line width=2.5pt, orange] (0.798858859,0.044360508) to (0.80976977,0.040964571);
        
            \draw [line width=2.5pt, orange] (0.820680681,0.037759456) to (0.831591592,0.034723568);
        
            \draw [line width=2.5pt, orange] (0.842502503,0.03183862) to (0.853413413,0.02908891);
        
            \draw [line width=2.5pt, orange] (0.864324324,0.026460775) to (0.875235235,0.023942173);
        
            \draw [line width=2.5pt, orange] (0.886146146,0.021522351) to (0.897057057,0.019191574);
        
            \draw [line width=2.5pt, orange] (0.907967968,0.016940892) to (0.918878879,0.014761934);
        
            \draw [line width=2.5pt, orange] (0.92978979,0.012646708) to (0.940700701,0.010587401);
        
            \draw [line width=2.5pt, orange] (0.951611612,0.008576145) to (0.962522523,0.006604728);
        
            \draw [line width=2.5pt, orange] (0.973433433,0.004664213) to (0.984344344,0.002744357);
            \addplot +[red, dashed, mark=none] coordinates {(0.667927928, -0.05) (0.667927928, 0.3)};
            \addplot +[red, dashed, mark=none] coordinates {(1, -0.05) (1, 0.3)};
            \node at (0.64, 0.25){$\mathcal{R}_1$};
            \node at (0.85,0.25){$\mathcal{R}_2$};
            \node at (1.05,0.25){$\mathcal{R}_3$};
            \end{axis}
        \end{tikzpicture}
     \end{subfigure}
    \begin{subfigure}[b]{0.45\textwidth}
         \begin{tikzpicture}[scale=0.7]
            \begin{axis}[
                title = $\kappa \text{ = } 0.3 \text{, } \theta \text{ = } 1.2$ ,
                axis lines=middle,
                cycle list name=black white,
                xmin=0.6, xmax=1.15,ymin=0,ymax=0.32,
                ytick={0.05,0.1,0.15,0.2,0.25},
                xlabel={$R_0(C^i)$},
                ylabel={$i^*$},
                ylabel near ticks,
                xlabel near ticks,
            ]
            \addplot+[
                blue,
                only marks,
                mark size=1.25pt]
            table[x=x, y=y]
            {Simulations_Data/dat6_estable.dat};

            \draw [line width=2.5pt, orange, dashed] (1.05,0.15) to (1.125,0.15);
            \node[below] at (1.0875,0.15){$\text{\tiny{Unstable}}$};
            
            \draw [line width=2.5pt, blue] (1.05,0.1) to (1.125,0.1);
            \node[below] at (1.0875,0.1){$\text{\tiny{Stable}}$};
            
            \draw [line width=2.5pt, orange] (0.649379379,0.143318808) to (0.66029029,0.119071094);

            \draw [line width=2.5pt, orange] (0.671201201,0.106111418) to (0.682112112,0.096313713);
        
            \draw [line width=2.5pt, orange] (0.693023023,0.088235043) to (0.703933934,0.081284676);
        
            \draw [line width=2.5pt, orange] (0.714844845,0.075149958) to (0.725755756,0.069640856);
        
            \draw [line width=2.5pt, orange] (0.736666667,0.064631229) to (0.747577578,0.060032036);
        
            \draw [line width=2.5pt, orange] (0.758488488,0.055777505) to (0.769399399,0.051817314);
        
            \draw [line width=2.5pt, orange] (0.78031031,0.048111872) to (0.791221221,0.044629305);
        
            \draw [line width=2.5pt, orange] (0.802132132,0.041343455) to (0.813043043,0.038232499);
        
            \draw [line width=2.5pt, orange] (0.823953954,0.03527796) to (0.834864865,0.032463991);
        
            \draw [line width=2.5pt, orange] (0.845775776,0.029776839) to (0.856686687,0.027204426);
        
            \draw [line width=2.5pt, orange] (0.867597598,0.02473603) to (0.878508509,0.02236202);
        
            \draw [line width=2.5pt, orange] (0.889419419,0.020073644) to (0.90033033,0.017862844);
        
            \draw [line width=2.5pt, orange] (0.911241241,0.015722089) to (0.922152152,0.01364422);
        
            \draw [line width=2.5pt, orange] (0.933063063,0.011622294) to (0.943973974,0.009649411);
        
            \draw [line width=2.5pt, orange] (0.954884885,0.007718512) to (0.965795796,0.005822116);
        
            \draw [line width=2.5pt, orange] (0.976706707,0.003951958) to (0.987617618,0.002098439);
            \addplot +[red, dashed, mark=none] coordinates {(0.648288288, -0.05) (0.648288288, 0.3)};
            \addplot +[red, dashed, mark=none] coordinates {(1, -0.05) (1, 0.3)};
            \node at (0.625,0.25){$\mathcal{R}_1$};
            \node at (0.85,0.25){$\mathcal{R}_2$};
            \node at (1.05,0.25){$\mathcal{R}_3$};
            \end{axis}
        \end{tikzpicture}
     \end{subfigure}\hspace{5mm}
    \begin{subfigure}[b]{0.45\textwidth}
         \begin{tikzpicture}[scale=0.7]
            \begin{axis}[
                title = $\kappa \text{ = } 0.01 \text{, } \theta \text{ = } 1.2$ ,
                axis lines=middle,
                cycle list name=black white,
                xmin=0.55,xmax=1.15,ymin=0,ymax=0.33,
                ytick={0.05,0.1,0.15,0.2,0.25},
                xlabel={$R_0(C^i)$},
                ylabel={$i^*$},
                ylabel near ticks,
                xlabel near ticks,
            ]
            \addplot+[
                blue,
                only marks,
                mark size=1.25pt]
            table[x=x, y=y]
            {Simulations_Data/dat7_estable.dat};

            \draw [line width=2.5pt, orange, dashed] (1.05,0.15) to (1.125,0.15);
            \node[below] at (1.0875,0.15){$\text{\tiny{Unstable}}$};
            
            \draw [line width=2.5pt, blue] (1.05,0.1) to (1.125,0.1);
            \node[below] at (1.0875,0.1){$\text{\tiny{Stable}}$};
            
            \draw [line width=2.5pt, orange] (0.618828829,0.150752227) to (0.62973974,0.128281976);

            \draw [line width=2.5pt, orange] (0.640650651,0.11509809) to (0.651561562,0.104980526);
        
            \draw [line width=2.5pt, orange] (0.662472472,0.096581831) to (0.673383383,0.089328682);
        
            \draw [line width=2.5pt, orange] (0.684294294,0.082911683) to (0.695205205,0.077140458);
        
            \draw [line width=2.5pt, orange] (0.706116116,0.071887544) to (0.717027027,0.067062383);
        
            \draw [line width=2.5pt, orange] (0.727937938,0.062597746) to (0.738848849,0.058441994);
        
            \draw [line width=2.5pt, orange] (0.74975976,0.054554372) to (0.760670671,0.050901998);
        
            \draw [line width=2.5pt, orange] (0.771581582,0.047457841) to (0.782492492,0.044199328);
        
            \draw [line width=2.5pt, orange] (0.793403403,0.041107352) to (0.804314314,0.038165537);
        
            \draw [line width=2.5pt, orange] (0.815225225,0.035359701) to (0.826136136,0.032677431);
        
            \draw [line width=2.5pt, orange] (0.837047047,0.030107765) to (0.847957958,0.027640931);
        
            \draw [line width=2.5pt, orange] (0.858868869,0.025268137) to (0.86977978,0.022981404);
        
            \draw [line width=2.5pt, orange] (0.880690691,0.020773412) to (0.891601602,0.01863738);
        
            \draw [line width=2.5pt, orange] (0.902512513,0.016566948) to (0.913423423,0.014556063);
        
            \draw [line width=2.5pt, orange] (0.924334334,0.012598877) to (0.935245245,0.010689628);
        
            \draw [line width=2.5pt, orange] (0.946156156,0.008822501) to (0.957067067,0.006991471);
        
            \draw [line width=2.5pt, orange] (0.967977978,0.005190083) to (0.978888889,0.003411146);
            \addplot +[red, dashed, mark=none] coordinates {(0.616646647, -0.05) (0.616646647, 0.31)};
            \addplot +[red, dashed, mark=none] coordinates {(1, -0.05) (1, 0.31)};
            \node at (0.58,0.25){$\mathcal{R}_1$};
            \node at (0.85,0.25){$\mathcal{R}_2$};
            \node at (1.05,0.25){$\mathcal{R}_3$};
            \end{axis}
        \end{tikzpicture}
     \end{subfigure}

    \caption{Equilibria points computed using $\theta = 1.2$ and we varying $\kappa = 0.8, 0.5, 0.3 $ and $0$. We can see that no $R_0$ allows us to obtain three possible equilibria points.}
    \label{fig_no_cubic}
\end{figure}
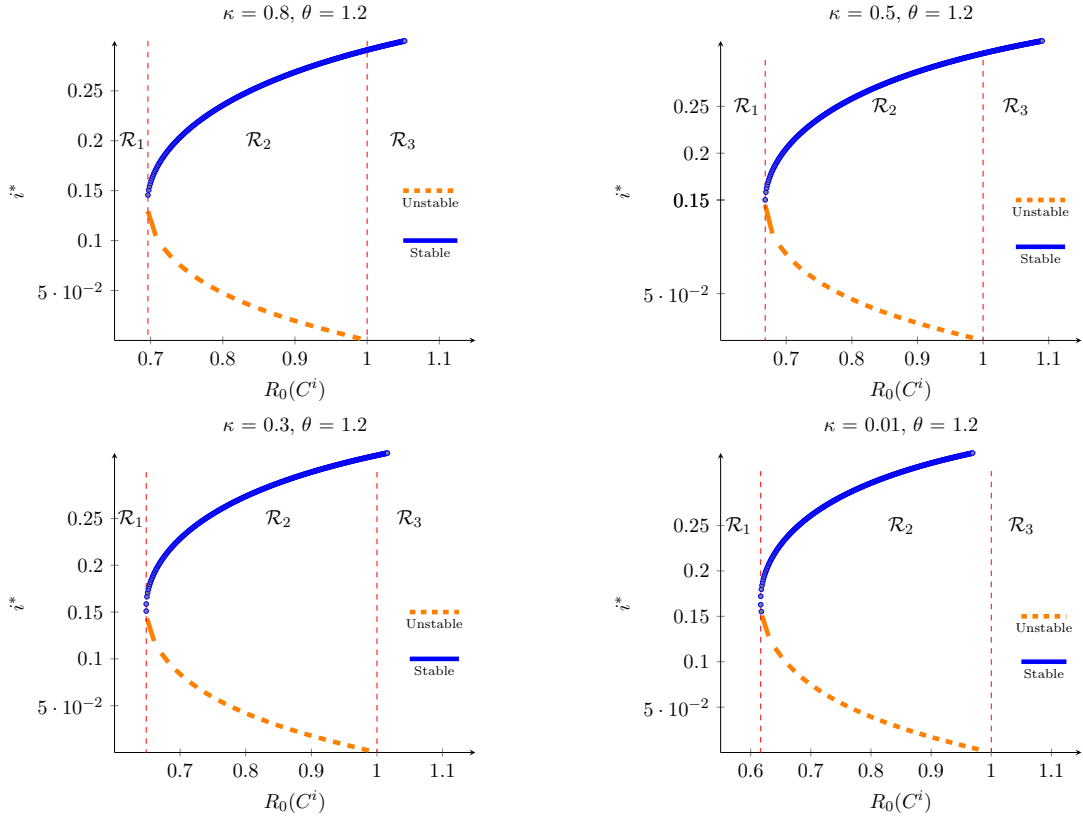

\begin{remark}
From the previous numerical results, we can derive the following conjectures:

\begin{itemize}
    \item[a)] A cubic bifurcation plot can be found for sufficiently high values of $\theta = \frac{C^{r}}{C^s}$, independently of $\kappa$. We observe an interval for $R_0$ in which those three equilibria points can be found, and it depends on $(\theta, \kappa)$.
    \item[b)] For cubic bifurcation plots, a sensible region $[1, 1+ \epsilon]$ could attain an endemic stable or a small non-stable equilibrium.
    \item[c)] When $\theta$ is small, there is no cubic behavior for any $\kappa$, and for all $R_0>1$, there is only one possibility for an endemic equilibrium.
\end{itemize}
\end{remark}

\subsection{Theoretical Results} 

Based on our previous simulations, we would like to formalize the conditions for the existence of regions that provide us with such cubic behavior. For that, we propose the following result.

\begin{theorem}\label{main_theorem}
Let $\mu,\gamma,\phi$ be positive real numbers. Define $R_\mu$ and $R_\phi$ as in (\ref{rphirmu}) and suppose that \begin{equation}\label{inequality}
    R_{\phi} > \frac{1+ R_\mu^2}{(1-R_\mu )^2}.
\end{equation} 
Then, there exist $0< \theta_1 <\theta_2$ such that for every $\theta \in [\theta_1,\theta_2]$ and every $\kappa \in [0,1]$, there is an $R_0>0$ such that the polynomial equation $a_0 + a_1 X + a_2 X^2 + a_3 X^3 = 0$ where $a_0, \cdots, a_3$ are defined by (\ref{coefficients}), has three distinct real roots in the  interval $[0,1]$. Moreover, for each pair $(\kappa, \theta)$, this hold for all $R_0$ in a neighborhood of the form $[1, 1+\epsilon({\theta})]$.
\end{theorem}

In other words, there is a range of the fraction $\theta = \frac{C^{r}}{C^s}$ which yields a cubic bifurcation plot, under our condition (\ref{inequality}), independently of the value of $\kappa = \frac{C^i}{C^s}$. We present the proof of this theorem, which uses the algebraic theory of Sturm chains. A preamble for this theory can be found in the Appendix.

\begin{proof}

Let us assume that the polynomial $f(X)$ has three different real roots. In this case, as discussed in proposition (\ref{sturm_norepeated_roots}), the sequence of higher derivatives of $f(X)$ forms a Sturm sequence on any interval. This sequence is then
\begin{align*}
    [a_0 + a_1 X + a_2 X^2 + a_3 X^3, 
    a_1 + 2a_2X + 3a_3X^2, 2a_2 + 6a_3X, 6a_3].
\end{align*}

Its values at $x=0$ and $x = 1$ are respectively $(a_0,a_1,2a_2,6a_3)$ and $(a_0 + a_1 + a_2 + a_3, a_1 + 2a_2 + 3a_3, 2a_2 + 6a_3, 6a_3)$. Our goal will be to find values for $R_0$ for which the signs of these sequences are $-,+,-,+$ at $x=0$ and $+,+,+,+$ at $x=1$. By Proposition (\ref{sturm_norepeated_roots}), this would prove the existence of three different roots in the interval $[0,1]$, since in this case $V_S(0)=3$ and $V_S(1)=0$. The reader can verify that for this to happen it is enough to have:
\begin{equation}\label{ineq_needed}
    a_0 < 0, \quad a_2 < 0, \quad a_3 > 0, \quad a_0 + a_1 > 0, \quad \text{ and } a_2 + a_3 > 0.
\end{equation}
Note that if $R_0>1$ then $a_0 < 0$ and if $R_0 > R_\mu$, then for all $\kappa \in [0,1]$ we have 
\begin{align*}
    a_3 = R_\phi^2 [ R_0 - R_\mu(1-\kappa)] > R_\phi^2 [ R_0 - R_\mu] > 0.
\end{align*}
Because $R_\mu < 1$, then $R_0>1$ is sufficient to ensure that $a_0 < 0$ and $a_3 > 0$. 

We now move to a coordinate plane with $R_0$ in the $x$-axis and $\theta$ in the $y$-axis. The coefficients $a_0,\ldots,a_3$ can be seen as linear equations in this plane, given by
\begin{align*}
    & a_3 = AR_0 - B(1-\kappa) \\
    & a_2 = C + DR_0 - E(1-\kappa) - F(1-\theta) \\
    & a_1 = G + HR_0 - I(1-\kappa) - J(1-\theta) \\
    & a_0 = K(1-R_0),
\end{align*}
where the constants $A,B,\cdots, K$ depend only on $R_\phi$ and $R_\mu$ and are given by
\begin{align*}
    & A := R_{\phi}^2, \quad B := R_{\mu}R_{\phi}^2, \quad C := R_{\phi}^2R_{\mu}, \quad D := R_{\phi}(1-R_{\phi}) + R_{\phi}R_{\mu}, \\
    & E := R_{\phi}R_{\mu}(1+R_{\mu}), \quad F := R_{\phi}R_{\mu}(1-R_{\mu}), \quad G := R_{\phi}R_{\mu}(1+R_{\mu})\\ 
    & H := R_{\mu}(1-R_{\phi}) - R_{\mu}R_{\phi}, \quad I := R_{\mu}^2, \quad J := R_{\mu}(1-R_{\mu}) \quad K := R_{\mu}^2.
\end{align*}
Note that all constants are positive, except perhaps $D$ and $H$  (we don't know the sign of $1 - R_\phi$). 

Geometrically, the inequalities in (\ref{ineq_needed}) refer to an area that is below the line $\ell_1 := \{a_2 = 0 \}$ and above the lines $ \ell_2 := \{a_2 + a_3 = 0\}$ and $\ell_3 : = \{a_0 + a_1 = 0\}$ in the $(R_0, \theta)$ plane. Consider the intersections of these three lines with the $R_0 = 1$ vertical line. If we prove that the $\theta$ coordinate of the intersection of the $\ell_1$ with this vertical axis is bigger than the $\theta$ coordinates of intersections of the other two ($\ell_2$ and $\ell_3$), then there would be an interval to the right of $R_0 = 1$ which is below the line $\ell_1$ and above both $\ell_1$ and $\ell_3$. See the next figure for a visual intuition.

\begin{figure}[H]
    \centering
    \begin{tikzpicture}
    \begin{axis}[
        ticks=none,
        axis y line=middle,
        axis x line=middle,
        every axis x label/.style={at={(current axis.right of origin)},anchor=west},
        every axis y label/.style={at={(current axis.north west)},above=2mm},
        minor tick num=2,
        xlabel=$R_0$,
        ylabel=$\theta$,
        xmin=-0.5,
        ymax= 7
    ]
    
    \addplot [
        domain=0:3, 
        samples=100, 
        color=blue,
        name path=A
        ]
        {5.7-1.5*x};
    \addlegendentry{\(\ell_1\)};
    
    \addplot [
        domain=0:3, 
        samples=100, 
        color=red,
        name path=B
    ]
    {3.6-2*x};
    \addlegendentry{\(\ell_2\)}
    
    \addplot [
        domain=0:3, 
        samples=100, 
        color=green,
        name path=C
        ]
        {-0.5+1.2*x};
    \addlegendentry{\(\ell_3\)};
    
    \draw [dashed] (1, -2.5) -- (1, 6);
    \node (B) at (1,-3) {$R_0 = 1$};
    
    \addplot +[mark=none] coordinates {(1, -3.5) (1, -4)};
    
    \path[draw,pattern=north east lines] (1,4.2) -- (1,1.65) -- (1.27,1.1) -- (2.3,2.26) ;
    
    \node (B) at (2.3,-3) {$R_0 = 1+ \epsilon_{\theta}$};
    \draw [dashed] (2.3, -2.5) -- (2.3, 6);
    \draw [solid] (2.3, -3.5) -- (2.3, -4);

    \draw [dotted] (1, 4.2) -- (0, 4.2);
    \node (B) at (-0.2,4.2) {$\theta_1^*$};
    
    \draw [dotted] (1, 1.6) -- (0, 1.6);
    \node (B) at (-0.2,1.6) {$\theta_2^*$};
    
    \draw [dotted] (1, 0.7) -- (0, 0.7);
    \node (B) at (-0.2,0.7) {$\theta_3^*$};

    \end{axis}
    \end{tikzpicture}
\end{figure}

This is a specific example, and we don't know the signs of the slopes of lines $\ell_1,\ell_2$, and $\ell_3$. However, their intersections at $R_0 = 1$ define the existence of a region below $\ell_1$ and above both $\ell_2$ and $\ell_3$ in a neighborhood of $R_0 = 1$.

The intersection of line $\ell_1$ at $R_0 = 1$ gives us $$\theta_1^* = 1 - \frac{C+D}{F} + \frac{E}{F}(1-\kappa),$$ and the intersection of line $\ell_2$ at $R_0=1$ gives us $$\theta_2^* = 1 - \frac{C+D}{F} + \frac{E}{F}(1-\kappa) + \frac{B(1-\kappa) - A}{F} < \theta_1^* + \frac{B-A}{F},$$ we note that $B-A = R_\mu R_\phi^2 - R_\phi^2 = R_\phi^2(R_\mu - 1) <0$, so we indeed have $\theta_2^* < \theta_1^*$.

The $\theta$ coordinate of the intersection between $\ell_3$ and $R_0 = 1$ is $$\theta_3^* = 1 - \frac{G+H}{J} + \frac{I}{J}(1-\kappa) < 1 - \frac{G+H}{J} + \frac{I}{J}.$$

We observe that $\theta_1^* > 1- \frac{C+D}{F}$ for all $\kappa \in [0,1]$, so we would need

\begin{align*}
    1- \frac{C+D}{F} > 1 - \frac{G+H}{J} + \frac{I}{J},
\end{align*}

to guarantee $\theta_3^* < \theta_1^*$ for all $\kappa$. This is equivalent to $(G+H-I)F > (C+D)J$, and by expanding this expression, we get the inequality (\ref{inequality}) from our statement. Note that this region can be obtained independently of $\kappa$. Inequality (\ref{inequality}) will imply that $\frac{C+D}{F}<1$ and $\frac{G+H-I}{J}<1$, and these are readily seen to imply that $\theta_i^*>0$ for $i=1,3$. Therefore the desired interval for $\theta$ can be taken between $\max\{\theta_3^*,\theta_2^*\}$ and $\theta_1^*$.

This proves that when (\ref{inequality}) is true, and our polynomial $f(X)$ doesn't have repeated roots, there is a region in the $(R_0,\theta)$ plane for which $f(X)$ has three real distinct roots in the interval $[0,1]$. 
\end{proof}

\begin{remark}
We note that both $R_0$ and the coefficients $a_0,\cdots,a_3$ are dependent, simultaneously, on the values of $C^h$. If we fix the values of $\kappa$ and $\theta$, all results will depend only on one of the $C^h$'s.
\end{remark}

In the next section, we will explore numerical results related to this theorem. Our examples concern the $C^i<C^s$ scenario ($\kappa<1$). In this case, we assume that infection decreases contacts, which is intuitive. The other option, $\kappa>1$, makes biological sense when infected populations are large because, under those circumstances, susceptible individuals may be overly cautious about engaging in contact with others \cite{CCast13}. The following theorem shows that the case $\kappa>1$ is more stable. Because of this, we focus on the $\kappa<1$ scenario from now on.

\begin{theorem}\label{nocycles}
    If $C^i>C^s$ then (\ref{model1}) has no limit cycles in the region $\{(i,r) \in \bR^2: i>0, r>0\}$.
\end{theorem}

\begin{proof}
We use a similar technique as in \cite{Blythe91}. Writing $s=1-i-r$, we obtain the two-variable system:
\begin{align}\label{model2D}
    & \frac{di}{dt}= g(1-i-r,i,r)\beta i(1-i-r) + \phi r i-(\gamma+\mu)i = g_1(i,r) \nonumber \\
    & \frac{dr}{dt}= \gamma i - \phi r i - \mu r = g_2(i,r).
\end{align}

Then we have that 
\begin{align*}
    &\frac{g_1(i,r)}{ir} = g(1-i-r,i,r)\beta \left(\frac{1-i-r}{r}\right) + \phi - (\gamma+\mu)\frac{1}{r},\\
    &\frac{g_2(i,r)}{ir} = \frac{\gamma}{r} - \phi  - \frac{\mu}{i}.\\
\end{align*}

This implies that
\begin{align*}
    &\frac{\partial}{\partial i}\left( \frac{g_1(i,r)}{ir} \right) + \frac{\partial}{\partial r}\left( \frac{g_2(i,r)}{ir}\right) \\
    &= 
    \left[\left(\frac{\partial g}{\partial i} - \frac{\partial g}{\partial s}\right)\beta \left(\frac{1-i-r}{r}\right)\right] -\frac{\beta g(1-i-r,i,r)}{r} - \frac{\gamma}{r^2}.
\end{align*}

This is negative when $\frac{\partial g}{\partial i} 
 - \frac{\partial g}{\partial s}<0$, and using the definition of $g(\cdot)$ given by (\ref{C_orig}), this is equivalent to $C^s<C^i$. Applying the Dulac criterion, we obtain the non-existence of limit cycles in the region $\{(i,r) \in \bR^2, i>0, r>0, i+r < 1 \}$ for this system of differential equations.
\end{proof}

\section{Numerical Results}\label{section4}

\subsection{Stable equilibrium points}

We explore the equilibrium results in a simulation of disease scenarios. Let us consider Figure (\ref{fig_cubic}) with the case $\kappa = 0.8$ and $\theta =1.7$ (and all the other model parameters as in that example). Using $C^i=3$ we obtain $R_0 \approx 1.01057$. This $R_0$ is found in the $\mathcal{R}_3$ region in the bifurcation plot. Solving the corresponding cubic equation, we obtain three possible theoretical equilibrium points: $i^* \in \{i_1^* = 0.004914, i_2^* = 0.010455, i_3^* = 0.238099\}$.

Of these possibilities, $i_3^*$ and $i_1^*$  are asymptotically stable equilibrium points. The system could converge to each point depending on its initial conditions. The middle point, which is unstable, actually works as a threshold value as solutions drift away from it. If we take initial conditions $a_0 = (S(0),I(0),R(0))$ with $N= S(0) + I(0) + R(0)$, and let $i(0) = \frac{I(0)}{N}$, then the highest equilibrium will attract all solutions when $i(0)>i_2^*$, otherwise it is the lowest equilibrium to which the system converges.

The following graphs show cases for convergence to each equilibrium point. Figure (\ref{eqpoints}) displays $i(t)$ through time using initial conditions $a_0 = (N- \rho N - 10, \rho N, 10)$, where $\rho \in [0,1]$ and $N=10000$. On the left are some simulations using $\rho>i_2^*$, in which the system converges to $i_3^*$, the highest equilibrium possible. On the right, a system is solved with $\rho<i_2^*$, where the final point obtained is $i_1^*$, although with a much slower convergence rate. We included the bifurcation plot on the left, highlighting the region of interest.

\begin{figure}[H]
    \hspace{-20mm}
    \begin{subfigure}[b]{0.25\textwidth}
         \begin{tikzpicture}[scale=0.7]
            \begin{axis}[
                title = $\kappa \text{ = } 0.8 \text{, } \theta \text{ = } 1.7$ ,
                axis lines=middle,
                cycle list name=black white,
                xmin=0.8,xmax=1.15,ymin=0,ymax=0.3,
                ytick={0.05,0.1,0.2},
                xlabel={$R_0(C^i)$},
                ylabel={$i^*$},
                ylabel near ticks,
                xlabel near ticks,
            ]
            \draw[color=gray!5, fill = gray!50] (1,0) rectangle (1.0125,0.27);
            \addplot+[
                blue,
                only marks,
                mark size=1.25pt]
            table[x=x, y=y]
            {Simulations_Data/dat0.dat};
            \addplot +[red, dashed, mark=none] coordinates {(0.854504505, -0.05) (0.854504505, 0.27)};
            \addplot +[red, dashed, mark=none] coordinates {(1, -0.05) (1, 0.27)};
            \addplot +[red, dashed,mark=none] coordinates {(1.0116, -0.05) (1.0116, 0.27)};
            \node at (0.83,0.28){$\mathcal{R}_1$};
            \node at (0.925,0.28){$\mathcal{R}_2$};
            \node at (1.01,0.28){$\mathcal{R}_3$};
            \node at (1.075,0.28){$\mathcal{R}_4$};

            \addplot +[green, dashed, mark=none, ultra thick] coordinates {(0,0.238099) (1.1, 0.238099)};
            \node at (1.1, 0.22){$i_3^*$};

            \addplot +[green, dashed, mark=none, ultra thick] coordinates {(0,0.01455) (1.03, 0.01455)};
            \node at (1.05, 0.017){$i_2^*$};

            \addplot +[green, dashed, mark=none, ultra thick] coordinates {(0,0.004914) (1.1, 0.004914)};
            \node at (1.12, 0.007){$i_1^*$};
            
            \end{axis}
        \end{tikzpicture}
     \end{subfigure}\hspace{20mm}
    \begin{subfigure}[b]{0.25\textwidth}
        \begin{tikzpicture}[scale=0.7]
            \begin{axis}[
                axis lines=middle,
                cycle list name=black white,
                ymax = 0.3,
                xmax = 65000,
                xmin=-6000,
                xlabel={$t$},
                ylabel={$i(t)$},
                ylabel near ticks,
                xlabel near ticks,
            ]
            \addplot+[
                color=violet,
                mark size=1.25pt,
                mark=*,
                mark options={fill=violet!50}
                ]
            table[x=t, y=y1]
            {Comparison_Data/plot_eq1.dat};
            \addplot+[
                color=blue,
                mark size=1.25pt,
                mark=*,
                mark options={fill=blue!50}
                ]
            table[x=t, y=y2]
            {Comparison_Data/plot_eq1.dat};
            \addplot+[
                color=cyan,
                mark size=1.25pt,
                mark=*,
                mark options={fill=cyan!50}]
            table[x=t, y=y3]
            {Comparison_Data/plot_eq1.dat};
            \addplot+[
                color=gray,
                mark size=1.25pt,
                mark=*,
                mark options={fill=gray!50}]
            table[x=t, y=y5]
            {Comparison_Data/plot_eq1.dat};
            \addplot+[
                color=orange,
                mark size=1.25pt,
                mark=*,
                mark options={fill=orange!50}]
            table[x=t, y=y7]
            {Comparison_Data/plot_eq1.dat};
            \addplot+[
                color=pink!75,
                mark size=1.25pt,
                mark=*,
                mark options={fill=pink!50}]
            table[x=t, y=y9]
            {Comparison_Data/plot_eq1.dat};
            \addplot +[red, dashed, mark=none] coordinates {(0,0.238099) (55000, 0.238099)};
            \node at (-3500, 0.23){$i_3^*$};

            \addplot +[violet, solid, mark=none, thick] coordinates {(32500,0.051) (37000,0.051)};
            \node[right] at (37500, 0.05){\tiny $\rho = 0.03574$};

            \addplot +[blue, solid, mark=none, thick] coordinates {(32500,0.066) (37000,0.066)};
            \node[right] at (37500, 0.065){\tiny $\rho = 0.0610$};

            \addplot +[cyan, solid, mark=none, thick] coordinates {(32500,0.081) (37000,0.081)};
            \node[right] at (37500, 0.08){\tiny $\rho = 0.0863$};

            \addplot +[yellow!50!black!80, solid, mark=none, thick] coordinates {(32500,0.096) (37000,0.096)};
            \node[right] at (37500, 0.095){\tiny $\rho = 0.1369$};

            \addplot +[orange, solid, mark=none, thick] coordinates {(32500,0.1101) (37000,0.1101)};
            \node[right] at (37500, 0.11){\tiny $\rho = 0.1875$};

            \addplot +[pink!75, solid, mark=none, thick] coordinates {(32500,0.1251) (37000, 0.1251)};
            \node[right] at (37500, 0.125){\tiny $\rho = 0.2381 \approx i_3^*$};
            \end{axis}
        \end{tikzpicture}
    \end{subfigure}\hspace{15mm}
    \begin{subfigure}[b]{0.25\textwidth}
        \begin{tikzpicture}[scale=0.7]
            \begin{axis}[
                axis lines=middle,
                xmin=-60000,
                xmax=650000,
                ymax=0.0106,
                xlabel={$t$},
                ylabel={$i(t)$},
                ylabel near ticks,
                xlabel near ticks,
            ]
            \addplot+[
                color=violet,
                mark size=1.25pt,
                mark=*,
                mark options={fill=violet!50}
                ]
            table[x=t, y=y1]
            {Comparison_Data/plot_eq2.dat};
            \addplot+[
                color=blue,
                mark size=1.25pt,
                mark=*,
                mark options={fill=blue!50}
                ]
            table[x=t, y=y2]
            {Comparison_Data/plot_eq2.dat};
            \addplot+[
                color=cyan,
                mark size=1.25pt,
                mark=*,
                mark options={fill=cyan!50}]
            table[x=t, y=y3]
            {Comparison_Data/plot_eq2.dat};
            \addplot+[
                color=gray,
                mark size=1.25pt,
                mark=*,
                mark options={fill=gray!50}]
            table[x=t, y=y5]
            {Comparison_Data/plot_eq2.dat};
            \addplot+[
                color=orange,
                mark size=1.25pt,
                mark=*,
                mark options={fill=orange!50}]
            table[x=t, y=y7]
            {Comparison_Data/plot_eq2.dat};
            \addplot+[
                color=pink!75,
                mark size=1.25pt,
                mark=*,
                mark options={fill=pink!50}]
            table[x=t, y=y9]
            {Comparison_Data/plot_eq2.dat};
            \addplot +[red, dashed, mark=none] coordinates {(0,0.004914) (550000,0.004914)};
            \node at (-35000, 0.004914){$i_1^*$};

            \addplot +[violet, solid, mark=none, thick] coordinates {(325000,0.0071) (370000,0.0071)};
            \node[right] at (375000, 0.007){\tiny $\rho = 0.00125$};

            \addplot +[blue, solid, mark=none, thick] coordinates {(325000,0.0076) (370000,0.0076)};
            \node[right] at (375000, 0.0075){\tiny $\rho = 0.0024$};

            \addplot +[cyan, solid, mark=none, thick] coordinates {(325000,0.0081) (370000,0.0081)};
            \node[right] at (375000, 0.008){\tiny $\rho = 0.0036$};

            \addplot +[yellow!50!black!80, solid, mark=none, thick] coordinates {(325000,0.0086) (370000,0.0086)};
            \node[right] at (375000, 0.0085){\tiny $\rho = 0.0058$};

            \addplot +[orange, solid, mark=none, thick] coordinates {(325000,0.0091) (370000,0.0091)};
            \node[right] at (375000, 0.009){\tiny $\rho = 0.0081$};

            \addplot +[pink!75, solid, mark=none, thick] coordinates {(325000,0.0096) (370000,0.0096)};
            \node[right] at (375000, 0.0095){\tiny $\rho = 0.010455 \approx i_2^*$};
            \end{axis}
        \end{tikzpicture}
    \end{subfigure}
    \caption{Convergence of $i(t)$ to the equilibrium point $i^*$. On the left, the bifurcation plot was obtained for this case, with region $\mathcal{R}_3$ highlighted. The center plot shows cases of convergence to the maximum possible equilibrium point within this region. This happens when the initially infected proportion is high enough. The plot on the right shows cases of convergence towards the smallest equilibrium point in this region, obtained for sufficiently small values of $i(0)$.}
    \label{eqpoints}
\end{figure}
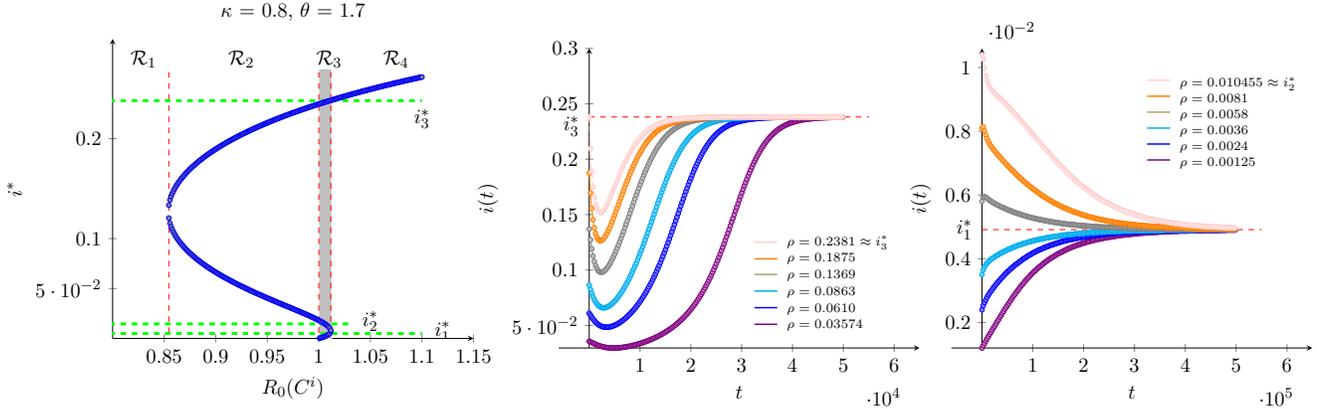

\subsection{Effect of \texorpdfstring{$(\kappa, \theta)$}{}}

Now we explore how the values of $\kappa$ and $\theta$ affect the size of the final equilibrium points \footnote{For each point on the grid, the convergence speed of the system varies. Each simulation was performed at a point in time when the difference between successive time states fell below a machine precision threshold.}. To compare within a given $R_0$, we fix $C^i=3$ (thus obtaining $R_0(C^i)$ as the examples above) and vary the values of $C^s$ and $C^{r}$ using $\kappa$ and $\theta$ respectively, we set $\kappa \in [0,1]$ and $\theta \in [0,2]$. Figure (\ref{kappa_theta_effect}) shows the effect of increasing both values on the stable steady states attained in the model for two different initial conditions. We observe that increasing $\kappa$ or $\theta$ yields a reduction in the final equilibrium of the system. However, we note that the value of this state is more sensible to $\theta$ than $\kappa$, indicating that, in the relapse case, contacts made by recovered individuals have a stronger impact on the disease outcome. Furthermore, we can see that high values of $(\kappa, \theta)$ may induce the system to attain a semi-disease-free steady state. This state, naturally, is more likely to be obtained when $i(0)$ is small, as seen by comparing both cases in Figure (\ref{kappa_theta_effect}). In other words, a considerable infected population makes this population more likely to become established.

\begin{figure}[H]
    \begin{subfigure}[b]{0.45\textwidth}
    \includegraphics[scale=0.4]{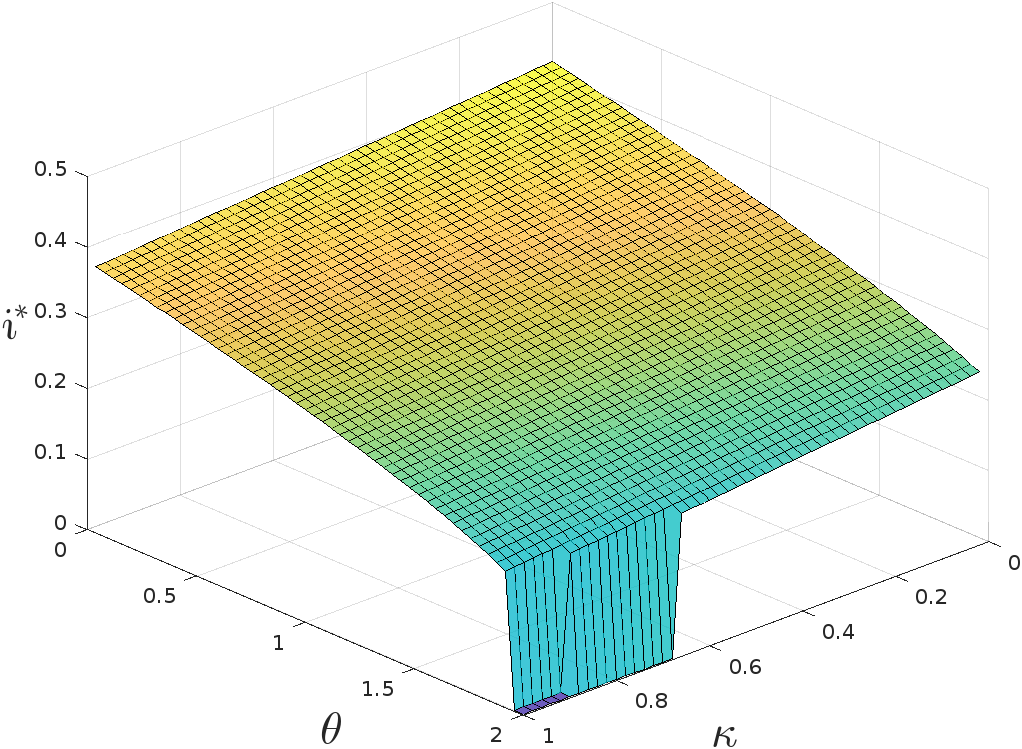}
    \end{subfigure}
    \begin{subfigure}[b]{0.45\textwidth}
    \includegraphics[scale=0.4]{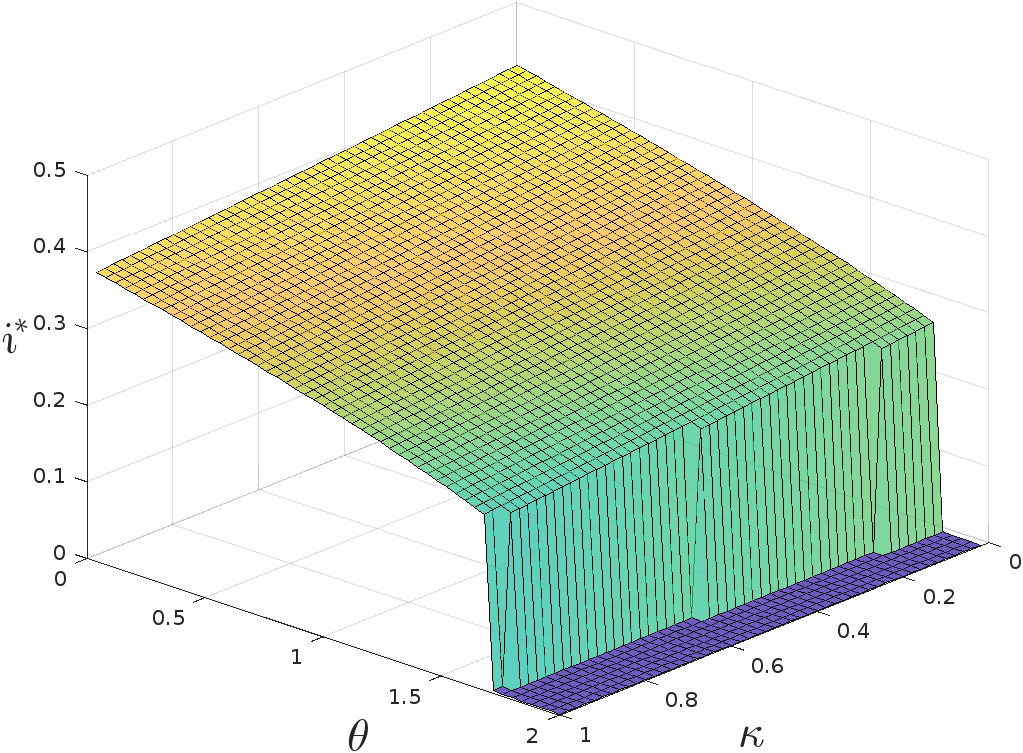}
    \end{subfigure}
    \caption{Effect of $(\kappa, \theta)$ for two different initial condition scenarios. On the left using $i(0)=0.1$, on the right $i(0)=0.02$. For low initial infected populations, a high value of $\theta$ yields disease eradication, independently of $\kappa$.}
    \label{kappa_theta_effect}
\end{figure}

Region $\mathcal{R}_3$ offers the most interesting behavior. In other regions, the disease is either maintained at a high steady prevalence or eradicated. In region $\mathcal{R}_3$, there is the possibility of a low equilibrium state, without achieving the disease disappearance from the population. However, the window for this behavior is small. For $\theta>\theta_1$ (as in Theorem \ref{main_theorem}), we find a window $[1, R_{0,\max}(\kappa, \theta)]$ which defines region $\mathcal{R}_3$, the next figure shows the upper limit of this interval, depending on $(\kappa, \theta)$. We can again infer a similar situation. This window becomes larger when $||(\theta, \kappa)||$ increases, however, the effect of $\theta$ is more prominent. In this case, $\theta_1 \approx 1.4$.

\begin{figure}[H]
    \centering
    \includegraphics[scale=0.75]{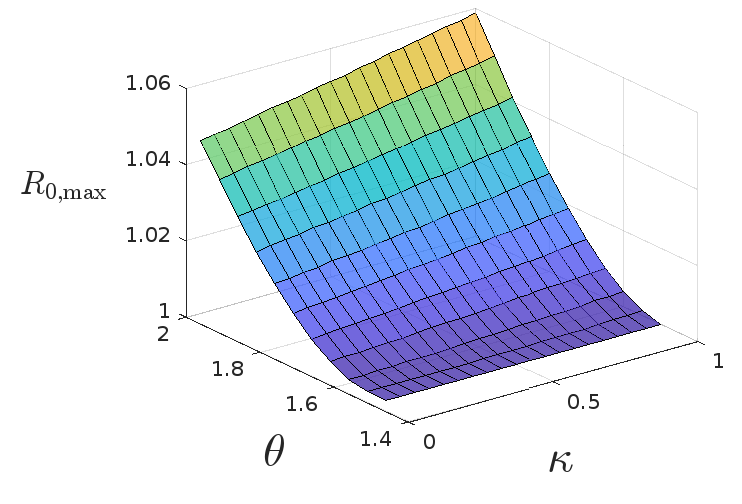}
    \caption{Effect of $(\kappa, \theta)$ on the $\mathcal{R}_3$ window length. When $\theta$ increases, there is more room to observe this region with more unstable behavior.}
    \label{R3_window_plot}
\end{figure}

\begin{remark}\label{size_of_beta}
Note that in these simulations, we decided to place less focus on the disease parameter $\beta$. The reason for this lies in the scale of the incidence rate function $g(\cdot)$. By using contact rates $C^i=3$ and $C^s = C^i/\kappa, C^{r}=C^i/\theta$, the scale forces us to reduce the value of $\beta$ to work with incidence rates that produce valuable epidemic scenarios. For example, if we use $\beta=0.0096$, as in \mbox{\cite{Arr22, Snchz07}}, this scenario would give us a basic reproductive number of $R_0 \simeq 10.105$ using these contacts, an exceedingly high and unrealistic number in many applications, which brings the model to a biological scenario with a single stable equilibrium point, resulting of an overestimation of the incidence term $\beta g(\cdot) S I /N$ of the system. Therefore, we note the importance of keeping in mind the scales of incidence rate functions when using the contact information in incidence rate functions for these models.
\end{remark}

\subsection{Other Examples}
In this subsection, we consider some numerical results and discussions on possible extensions of the contact rate disaggregation approach taken in the present study.

\begin{example}[Disaggregated contacts for Influenza, a non-relapse case]\label{influenza-example} We study the effect of disaggregated contact rates in epidemics models for individual-based transmitted diseases such as influenza. In this example, we use parameters for influenza transmission based on estimations performed in \cite{Chow08}. We then consider model \ref{model1} with the following epidemic parameters.

\begin{align*}
    &\bullet \beta = 0.07943065, \\
    &\bullet \gamma = 0.243902 =1/4.1\text{ (equivalent to }{4.1} \text{ recovery period as mentioned in \cite{Chow08}}), \\
    &\bullet \phi=0, \quad  \text{and } \mu = 0.0005.
\end{align*}

The estimation of the $\beta$ infection parameter was performed as follows. For influenza, \cite{Chow08} obtains a natural reproduction number estimation, $R_p \simeq 1.3$ for influenza seasons in different countries from 1972-1997. The number $R_p$ is defined as $R_p:= R_0(1-p)$ where $p$ is assumed to be a proportion of susceptible individuals that have been successfully immunized before an epidemic. We perform simulations with $R_0$ obtained using $p=0.2$, and $C^i=5$, giving $R_0 \approx 1.625$.

Although this model presents a non-relapse scenario, we can still incorporate the contact information and obtain similar numerical results as before. For example, as observed in Figure \ref{influenza_peak_prevalence}, the effect of $\theta$ over the peak epidemic prevalence of the model seems to be stronger than the effect of $\kappa$, thus indicating a similar behavior in the non-relapse case in terms of the contact proportions $\kappa, \theta$. In this non-relapse scenario, our focus shifts towards the peak prevalence, as the final equilibrium will have a null infected population.

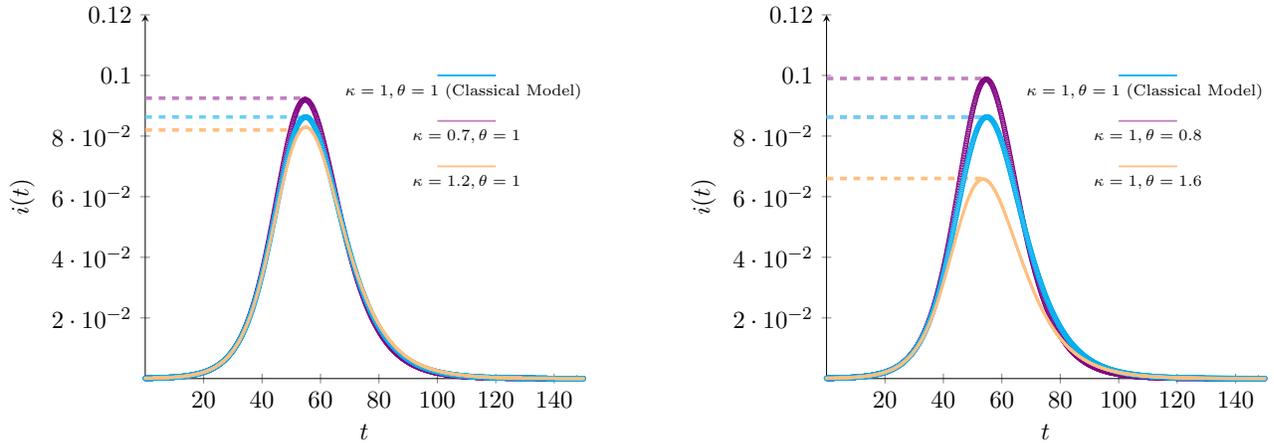
\begin{figure}[h]
    \centering
    \begin{subfigure}[b]{0.45\textwidth}
        \begin{tikzpicture}[scale=0.85]
            \begin{axis}[
                axis lines=middle,
                cycle list name=black white,
                ymax = 0.12,
                xlabel={$t$},
                ylabel={$i(t)$},
                ylabel near ticks,
                xlabel near ticks,
            ]
            \addplot+[
                color=violet,
                mark size=1pt,
                mark=*,
                mark options={fill=violet!50}
                ]
            table[x=t, y=y4]{Comparison_Data/dat_influenza_1.dat};
            \addplot +[violet!50, dashed, ultra thick, mark=none] coordinates {(0,0.0925) (55, 0.0925)};
            \addplot +[violet!50, solid, mark=none, thick] coordinates {(100,0.085) (120,0.085)};
            \node[below] at (110, 0.085){\tiny $\kappa = 0.7, \theta = 1$};
            \addplot+[
                color=cyan,
                mark size=1pt,
                mark=*,
                mark options={fill=cyan!50}]
            table[x=t, y=y7]
            {Comparison_Data/dat_influenza_1.dat};
            \addplot +[cyan!50, dashed, ultra thick, mark=none] coordinates {(0,0.086291) (55, 0.086291)};
            \addplot +[cyan, solid, mark=none, thick] coordinates {(100,0.1) (120,0.1)};
            \node[below] at (110, 0.1){\tiny $\kappa = 1, \theta = 1 \text{ (Classical Model) }$ };
            \addplot+[
                color=orange!50,
                mark size=0.5pt,
                mark=*,
                mark options={fill=orange!50}]
            table[x=t, y=y9]
            {Comparison_Data/dat_influenza_1.dat};
            \addplot +[orange!50, dashed, ultra thick, mark=none] coordinates {(0,0.082) (55, 0.082)};
            \addplot +[orange!50, solid, mark=none, thick] coordinates {(100,0.07) (120,0.07)};
            \node[below] at (110, 0.07){\tiny $\kappa = 1.2, \theta = 1$};
            \end{axis}
        \end{tikzpicture}
    \end{subfigure}\hspace{15mm}
    \begin{subfigure}[b]{0.45\textwidth}
        \begin{tikzpicture}[scale=0.85]
            \begin{axis}[
                axis lines=middle,
                cycle list name=black white,
                ymax = 0.12,
                xlabel={$t$},
                ylabel={$i(t)$},
                ylabel near ticks,
                xlabel near ticks,
            ]

            \addplot+[
                color=violet,
                mark size=1pt,
                mark=*,
                mark options={fill=violet!50}
                ]
            table[x=t, y=y1]{Comparison_Data/dat_influenza_1.dat};
            \addplot +[violet!50, dashed, ultra thick, mark=none] coordinates {(0,0.099) (55, 0.099)};
            \addplot +[violet!50, dashed, ultra thick, mark=none] coordinates {(0,0.086291) (55, 0.086291)};
            \addplot +[violet!50, solid, mark=none, thick] coordinates {(100,0.085) (120,0.085)};
            \node[below] at (110, 0.085){\tiny $\kappa = 1, \theta = 0.8$};
            \addplot+[
                color=cyan,
                mark size=1pt,
                mark=*,
                mark options={fill=cyan!50}]
            table[x=t, y=y3]
            {Comparison_Data/dat_influenza_2.dat};
            \addplot +[cyan!50, dashed, ultra thick, mark=none] coordinates {(0,0.086291) (55, 0.086291)};
            \addplot +[cyan, solid, mark=none, thick] coordinates {(100,0.1) (120,0.1)};
            \node[below] at (110, 0.1){\tiny $\kappa = 1, \theta = 1 \text{ (Classical Model) }$};

            \addplot+[
                color=orange!50,
                mark size=0.5pt,
                mark=*,
                mark options={fill=orange!50}]
            table[x=t, y=y9]
            {Comparison_Data/dat_influenza_2.dat};
            \addplot +[orange!50, dashed, ultra thick, mark=none] coordinates {(0,0.066) (55, 0.066)};
            \addplot +[orange!50, solid, mark=none, thick] coordinates {(100,0.07) (120,0.07)};
            \node[below] at (110, 0.07){\tiny $\kappa = 1, \theta = 1.6$};
            \end{axis}
        \end{tikzpicture}
    \end{subfigure}
    \caption{Different infected results varying both contact proportions $\kappa$ and $\theta$. We keep $\theta=1$ on the left and vary the proportion $\kappa$. Conversely, we keep $\kappa=1$ on the right and vary the value of $\theta$. We observe a bigger effect on the peak prevalence obtained in the figure on the right, that is, varying $\theta$.}
    \label{influenza_peak_prevalence}
\end{figure}

We expand these results for several combinations of values for $(\kappa, \theta)$ in the surface plot in Figure (\ref{influenza_peak_prevalence_heatmap}). We obtain a parallel situation as in the relapse simulations: we observe a bigger slope in the $\theta$ axis than in the $\kappa$ axis.

\begin{figure}[H]
    \centering
    \includegraphics[scale=0.8]{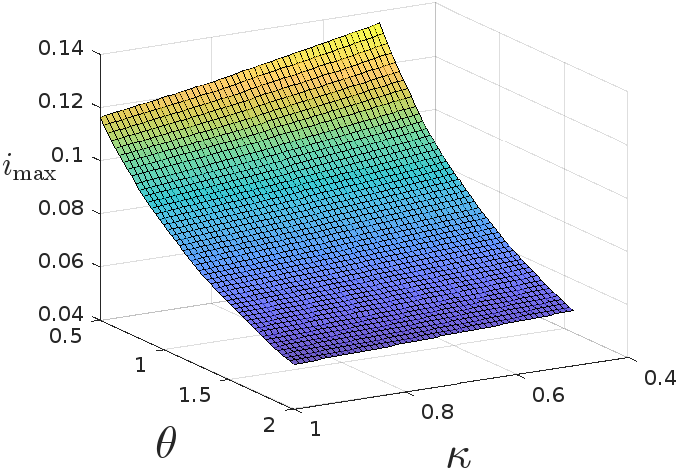}
    \caption{Disease peak infected prevalence varying both contact proportions $\kappa$ and $\theta$ for the influenza simulation example.}
    \label{influenza_peak_prevalence_heatmap}
\end{figure}

We note that, in general, a reduction in peak prevalence is expected using the disaggregated contact rates in comparison to the classical SIR model. For the non-relapse case, this was pointed out within simulations performed in \cite{Feni11} and also expanded by the results of this example; for the non-relapse case this observation has been supported by our simulations in the previous sections.
\end{example}

\begin{example}[Disaggregated contacts approach for more complex models: Discussion]\label{dengue-example} 

Incidence rate functions control the entrance of susceptible individuals into the infected population. In the classical SIR model, this occurs because of contact with other infected individuals. This is reflected in the numerator of the incidence rate function, being $C^sC^iN$, consisting of the susceptible contacts multiplied by the contribution of infected contacts\footnote{We could also argue that a better formulation for this denominator would be $C^sC^i I$, by reducing the effect only to consider the volume of the infected population. Looking at equations \ref{Balt1_incidence_rate_functions}, we see that \cite{Balt1} follows this idea. We decided instead to use the first formulation as it was the approach taken in the original adaptive setting reference: \cite{Feni11}.}. The denominator constitutes the total population activity, given by $\sum_h C^h h$, where the sum is over all possible health statuses, this yields the incidence rate as a proportion of the contact activity of the infected in terms of the total contact activity.

When adding new compartments or modifying the formulation of the model and trying to use the disaggregated contact approach discussed in the present study, this incidence rate formula should follow the same pattern: for the numerator to consider the contact information of compartments that might cause infection to susceptible individuals, and the denominator to reflect the total activity. We discuss some cases of how to apply this method to introduce such functions into more general epidemiological models.

First, consider the model proposed in \cite{Balt1}. Here, the authors provide a modified SIR model of risk-taker ($S_1$) and risk-evader ($S_2$) susceptible to study COVID-19 epidemic scenarios. This non-relapse epidemic model was fitted using the adaptive approach, which implies the use of contact-disaggregated incidence rate functions. For each group of susceptible, the infection might come as a result of contact either with infectious exposed -both risk-takers and risk-evaders ($E_1, E_2$)-, infectious asymptomatic -both risk-takers and risk-evaders ($A_1, A_2$)- and infected symptomatic ($I$). Therefore, the incidence rate for both susceptible compartments is given by 
\begin{align}\label{Balt1_incidence_rate_functions}
    & g_1(\cdot) := C^{S_1}\frac{\rho(C^{E_1}E_1 + C^{E_2}E_2) + \alpha(C^{A_1}A_1 + C^{A_2}A_2) + C^I I}{\sum_h C^h h} \nonumber \\
    & g_2(\cdot) := C^{S_2}\frac{\rho(C^{E_1}E_1 + C^{E_2}E_2) + \alpha(C^{A_1}A_1 + C^{A_2}A_2) + C^I I}{\sum_h C^h h},
\end{align}
where $h \in \{S_1,S_2,E_1,E_2,A_1,A_2,I,R\}$ (the model also has a recovered compartment, naturally), $\epsilon \in (0,1)$ is a reduction in infectious chances by taking a risk-evader approach, and $\rho, \alpha$ are reduction constants for non-symptomatic infectious populations. This example gives us an application of the abovementioned principle in constructing contact-based incidence rate functions.

Another interesting example of the application of this principle consists of vector-borne diseases. In these cases, susceptible humans become infected not by contact with other infected human individuals, but rather by contact with infected vectors. Furthermore, susceptible vectors become infected by contact with infected humans. This dynamic requests a change of form in the incidence rate functions. Let us take, for example, the dengue-chikungunya vector-borne epidemic model proposed in \cite{Sanchez18}. This model considers two populations: hosts ($h$) and vectors ($v$), and it is based on the following system of differential equations.

\begin{center}
\begin{tabular}{ p{7cm} | p{7cm}}
    Hosts & Vectors \\
    \hline
    $\frac{dS_h}{dt} = \mu_h N_h - \beta g_h(\cdot) S_h \frac{I_v}{N_v} - \mu_h S_h$ & $\frac{dS_v}{dt} = \mu_v N_v - \beta_v g_v(\cdot) S_v \frac{I_h}{N_h} - \mu_v S_v$ \\
    $\frac{dE_h}{dt} = \beta g_h(\cdot)S_h \frac{I_v}{N_v} - (\mu_h+\alpha_h) E_h$ & $ \frac{dE_v}{dt} = \beta g_v(\cdot) S_v \frac{I_h}{N_h} - (\mu_v+ \alpha_v) E_v$ \\
    $\frac{dI_h}{dt} = \alpha_h E_h - (\mu_h+\gamma) I_h$ & $\frac{dI_v}{dt} = \alpha_v E_v - \mu_v I_v$ \\
    $\frac{dR_h}{dt} = \gamma I_h - \mu_h R_h$ & 
    \end{tabular}
\end{center}

Here, the host population has five health classes: $S_h$, susceptible hosts, $E_h$, exposed hosts, $I_h$, infected hosts, and $R_h$, recovered hosts, and the vector population has three: $S_v$, susceptible vectors, $E_v$, exposed/latent vectors, and $I_v$, infected vectors. Total populations are $N_h$ for hosts and $N_v$ for vectors. There are no recovered vectors, as they die with the disease. For each population, we inserted incidence rate functions $g_h(\cdot), g_v(\cdot)$, which are constant and equal to $1$ in \cite{Sanchez18}.

Following the abovementioned principle, we can propose the following formulas for both functions.

\begin{align}\label{incidence_for_vector_borne}
    & g_h(\cdot) := \frac{C^{S_h}C^{I_v}I_v}{\sum_j C^j j} \nonumber \\
    & g_v(\cdot) := \frac{C^{S_v}C^{I_h}I_h}{\sum_j C^j h},
\end{align}
where $j \in \{S_h, E_h, I_h, R_h, S_v, E_v, I_v\}$ goes through all possible health statuses. This formulation considers the infection dynamics of vector-borne diseases: hosts become infected after contact with infected vectors, and vectors become infected after contact with infected hosts. Equations \ref{incidence_for_vector_borne} offer an alternative for researching the possible impact of non-linear incidence rate functions in more complex scenarios, such as vector-borne diseases. Introducing contact rates between susceptible hosts and infected vectors could offer a mathematical approach to model the interactions between these two populations and understand further indirect contact-impacting measures, such as protection against vectors. We believe this offers an opportunity to further understand the dynamics of this biological scenario, especially in light of real data-based analysis, such as performed in \cite{Sanchez18}.
\end{example}

\section{Discussion}\label{section5}

Motivated by the recent advances in the adaptive setting framework, we proposed a model incorporating non-linear relapse and contact behavior among individuals of different health classes. Our study aimed to explore the analytical properties of this model and investigate the effects of disaggregating contact rates on disease dynamics. We found that the model exhibits a high sensitivity to initial conditions and the relationships between contact rates, with significant implications for disease control strategies.

To gain insights into the behavior of our model, we performed numerical simulations that revealed several important features. First, we observed that the model's dynamics are highly dependent on the values of the basic reproductive number ($R_0$), which reflects the behavior of infected individuals. We established explicit conditions for multiple stable infected populations, which are highly sensitive to the model's initial conditions. Furthermore, we found that the impact of the contact rates for recovered individuals with relapse ($\theta$) is more substantial than that of infected individuals ($\kappa$), with larger differences required to achieve complete disease control for higher initial epidemic volumes.


Models incorporating relapse phenomena highlight the significant impact of recovered individuals on the progress of diseases. Such models exhibit different dynamics compared to those without relapse, with distinct results regarding recovery and relapse. Our study supports this view, with our conclusions showing that changes in the contact behavior of recovered individuals ($\theta$) have a more substantial effect on epidemic equilibria than corresponding changes in contact rates for infected individuals ($\kappa$), after normalizing with respect to the susceptible contact rate. The differences in behavior when recovering from the disease (or addiction) play a crucial role in determining the prevalence of the disease. Our findings suggest that a low $\theta$ value, indicating a lack of meaningful contact engagement by recovered individuals after infection, can establish a considerable epidemic burden. This highlights the importance of successfully reintegrating recovered individuals into society, which can reduce the likelihood of significant epidemics. Similar conclusions regarding the impact of recovered individuals on bifurcation plots have been observed in other relapse models \cite{Tasman22}.


Our results are closely tied to the behavior of infected individuals, as captured by the basic reproductive number $R_0$. We established explicit conditions for the existence of a region $1<R_0<R_{0,\max}$, characterized by multiple stable infected populations, which are highly sensitive to the model's initial conditions. We also found that this region becomes wider as the contacts of recovered individuals with relapse increase. In our simulations, we observed that the impact of $(\kappa,\theta)$ is intertwined with the initial infected population size, with larger initial epidemic volumes requiring more significant differences in contact rates to achieve disease control.



Building on the mathematical analysis presented in \cite{Arr22}, we confirmed the conclusions and discussions regarding the influence of recovered individuals on the prevalence of diseases with relapse.


Incorporating non-linear relapse significantly alters the dynamics of the SIR model, leading to more complex equilibria and bifurcation considerations. Our analysis adopted a non-linear relapse formulation (\ref{C_orig}) that assumes fixed contact rates among health compartments. We aimed to obtain analytical results that can be compared to future studies using a complete adaptive formulation. Such comparisons will be made against the non-linear non-adaptive model proposed in this article.

Our study underscores the importance of incorporating adaptive behavior and contact heterogeneity into epidemiological models, particularly in the presence of relapse phenomena. The results can inform public health policy decisions and provide a foundation for future research into the behavior of complex disease systems.

\section*{Acknowledgments}
The authors would like to thank the support from the Research Center in Pure and Applied Mathematics
and the Department of Mathematics at Universidad de Costa Rica.
\section*{Conflict of interest}
All authors declare no conflicts of interest in this paper.

\section{Appendix}\label{section6}

We use the general theory of Sturm chains to prove the theorem (\ref{main_theorem}). Here we set some basic definitions and properties. We base this treatment on the theory detailed in \cite{Eiser08}.

\begin{definition}
A sequence $ S = \{ p_0(x), p_1(x), p_2(x), \cdots, p_n(x) \}$ of polynomials in $\bR[x]$ is called a \textbf{Sturm chain} with respect to an interval $I$ if it satisfies the Sturm property:

\begin{center}
    If $\alpha \in I$ is a real root of $p_i(x)$, for some $i$ with $0 < i < n$. Then $p_{i-1}(\alpha)p_{i+1}(\alpha) < 0$. 
\end{center}

\end{definition}


\begin{definition}
    Let $f$ be a real rational function. The Cauchy index of $f$ at $x$ is defined by 
$$
        \Ind_x(f) := \Ind_x^+(f) - \Ind_x^-(f), \quad \text{where for} \sigma \in \{+,-\} \text{we have}
 $$
 \begin{align*}
        \Ind_x^{\sigma}(f): = \begin{cases}
            +\frac{1}{2}, & \text{ if } \displaystyle \lim_{y \to x^{\sigma}}f(y) = +\infty, \\
            -\frac{1}{2}, & \text{ if } \displaystyle  \lim_{y \to x^{\sigma}}f(y) = -\infty, \\
            0, & \text{ otherwise.} 
        \end{cases}
    \end{align*}
 
    The Cauchy index of $f$ at $[a,b]$ is then given by   
    \begin{equation}\label{Cauchyindex}
        \Ind_a^b(f) : = \Ind_a^+(f) - \Ind_b^-(f) + \sum_{x \in ]a,b[ } \Ind_x(f).
    \end{equation}

\end{definition}

\begin{remark}
    We can assume that $f$ is in its reduced form. That is, the numerator and denominator have no common factors. In this case, $\Ind_x(f)$ is non-zero (with values $1$ or $-1$) only for odd-multiplicity roots of the denominator, and since there are only finitely many such points, the sum in \ref{Cauchyindex} is well defined.
\end{remark}

We state the following generalization of the classical \textit{Sturm Theorem}.

\begin{theorem}[\cite{Eiser08}, Theorem 3.11]\label{Sturm}
    If $S = \{p_0(x),p_1(x),...,p_{n-1}(x),p_n(x)\}$ is a Sturm chain in $\mathbb{R}[x]$ with respect to $[a,b]$, then
    \begin{equation}\label{Sturmequality}
    \Ind_a^b\left(\frac{p_1}{p_0}\right) +  \Ind_a^b\left(\frac{p_{n-1}}{p_n}\right) = V_S(a) - V_S(b) .
    \end{equation}
    Where $V_S(c)$ is the number of sign changes in the values of consecutive polynomials of the chain $S$ at a point $x=c$; that is, the number of those $j \in \{1, \ldots, n \}$ for which $p_{j-1}(c) p_j(c) < 0$.
\end{theorem}

This theorem can be applied to a special case of non-repeated roots, but first, we need the following auxiliary lemma, which standard calculus arguments can prove.

\begin{lemma}\label{rootsofderivatives}
Let $p_0(x)$ be a polynomial of degree $n$ with $n$ distinct real roots.  Suppose that $c \in \bR$ such that $p'(c) = 0$, then $p(c)p''(c) < 0$.
\end{lemma}

\begin{proposition}\label{sturm_norepeated_roots}
    Let $p_0(x)$ be a polynomial of degree $n$ that has $n$ distinct real roots, then the sequence $ S_0 = \{p_0(x), p_0'(x), p_0^{(2)}(x) , p_0^{(3)}(x), \cdots, p_0^{(n)}(x)\}$ of higher derivatives of $p_0(x)$ is a Sturm chain with respect to any interval $[a,b]$.  Moreover, if $p_0(a)p_0(b) \neq 0$, then the number of roots of $p_0(x)$ in $[a,b]$ equals $V_{S_0}(a)-V_{S_0}(b)$.
\end{proposition}

\begin{proof}
    By a repeated application of Lemma \ref{rootsofderivatives}, it is easy to check that the sequence $S_0$ of higher derivatives of $p_0(x)$ is a Sturm chain with respect to any interval $I$. Given that $p_0^{(n)}$ is a non-null constant, then $\frac{p_0^{(n-1)}}{p_0^{(n)}}$ is a polynomial which means that $\Ind_a^b\left(\frac{p_0^{(n-1)}}{p_0^{(n)}}\right)=0$. For the first term in \ref{Sturmequality}, if we write $p_0(x) = c (x-a_1)\cdots(x-a_n)$, where $a_1 , \ldots a_n$ are the roots of $p_0$, we have
    $$
    \frac{p_1(x)}{p_0(x)} = \frac{p_0'(x)}{p_0(x)} = \frac{ c \sum_{i=1}^n \prod_{j \neq i} (x-a_j)}{ c \prod_{j=1}^n (x-a_j) } = \frac{1}{x - a_1} + \frac{1}{x - a_2} + \ldots \frac{1}{x - a_n}. 
    $$
    So, for each $j = 1, \ldots, n$,  $\lim_{x \to a_j^{\pm}} = \pm \infty$, that is $\Ind_{a_j}\left(\frac{p_0'}{p_0}\right) = 1$. Therefore, the Cauchy index $\Ind_a^b\left(\frac{p_0'}{p_0}\right)$ is the number of roots of $p_0$ that are contained in $[a,b]$, and by Theorem \ref{Sturm}, this coincides with $V_{S_0}(a)-V_{S_0}(b)$.
\end{proof}


\end{document}